\pdfoutput=1
\documentclass[a4paper,abstract=true,DIV=default,toc=index]{scrartcl}
\KOMAoption{draft}{true}
\KOMAoption{overfullrule}{true}
\KOMAoption{twoside}{true}

\usepackage[normalem]{ulem}

\usepackage[hyperpageref]{backref}

\usepackage{amsmath}
\usepackage{amsthm}
\usepackage{amsfonts}
\usepackage{color,enumitem,graphicx}

\usepackage[english]{babel}
\usepackage[dvipsnames]{xcolor}
\usepackage{hyperref}

\newcommand\myshade{85}
\colorlet{mylinkcolor}{violet}
\colorlet{mycitecolor}{YellowOrange}
\colorlet{myurlcolor}{Aquamarine}

\hypersetup{
	linkcolor  = mylinkcolor!\myshade!black,
	citecolor  = mycitecolor!\myshade!black,
	urlcolor   = myurlcolor!\myshade!black,
	colorlinks = true,
}

\usepackage[T1]{fontenc}
\usepackage{libertinus,libertinust1math}
\usepackage[
    protrusion=true,
    activate={true,nocompatibility},
    final,
    tracking=true,
    kerning=true,
    spacing=true,
    factor=1100]{microtype}
\SetTracking{encoding={*}, shape=sc}{40}

\usepackage[en-us]{datetime2}

\usepackage{fixmath}

\numberwithin{equation}{section}

\newtheorem{theorem}{Theorem}[section]
\newtheorem{lemma}[theorem]{Lemma}
\newtheorem{proposition}[theorem]{Proposition}
\newtheorem{corollary}[theorem]{Corollary}
\theoremstyle{definition}

\theoremstyle{remark}
\newtheorem{remark}[theorem]{Remark}

\newcommand{\R}{\mathbb{R}}

\usepackage{todonotes}

\numberwithin{equation}{section}

\DeclareSymbolFont{rsfs}{U}{rsfs}{m}{n}
\DeclareSymbolFontAlphabet{\mathscr}{rsfs}

\begin{document}

\title{On Critical Kirchhoff problems driven by the fractional Laplacian\footnote{The first and third author are supported by
    GNAMPA, project ``Equazioni alle derivate parziali: problemi e
    modelli.''}}

\author{Luigi Appolloni\thanks{Dipartimento di Matematica e
    Applicazioni, Universit\`a degli Studi di Milano Bicocca. Email:
    \href{mailto:l.appolloni1@campus.unimib.it}{l.appolloni1@campus.unimib.it}}
  \and Giovanni Molica Bisci\thanks{Dipartimento di Scienze Pure e
    Applicate, Universit\`{a} di Urbino Carlo Bo. Email:
    \href{mailto:giovanni.molicabisci@uniurb.it}{giovanni.molicabisci@uniurb.it}}
  \and Simone Secchi\thanks{Dipartimento di Matematica e Applicazioni,
    Universit\`a degli Studi di Milano Bicocca. Email:
    \href{mailto:simone.secchi@unimib.it}{simone.secchi@unimib.it}}}

\date{\DTMnow}

\maketitle

\begin{abstract}
We study a nonlocal parametric problem driven by the fractional Laplacian operator combined with a Kirchhoff-type coefficient and involving a critical nonlinearity term in the sense of Sobolev embeddings. Our approach is of variational and topological nature. The obtained results can be viewed as a nontrivial extension to the nonlocal setting of some recent contributions already present in the literature.
\end{abstract}

\section{Introduction} \label{sec:1}

The equation that goes under the name of \emph{Kirchhoff equation} was
proposed in \cite{zbMATH02674345} as a model for the transverse
oscillation of a stretched string in the form
\begin{gather} \label{eq:kirchhoff} \rho h \, \partial^2_{tt} u -
  \left( p_0 + \frac{\mathcal{E}h}{2L} \int_0^L \left| \partial_x u
    \right|^2 \, dx \right) \partial^2_{xx} u + \delta \, \partial_t u
  + f(x,u) =0
\end{gather}
for $t \geq 0 $ and $0<x<L$, where $u=u(t,x)$ is the lateral
displacement at time $t$ and at position $x$, $\mathcal{E}$ is the
Young modulus, $\rho$ is the mass density, $h$ is the cross section
area, $L$ the length of the string, $p_0$ is the initial stress
tension, $\delta$ the resistance modulus and $g$ the external
force. Kirchhoff actually considered only the particular case of
\eqref{eq:kirchhoff} with $\delta=f=0$.

Through the years, this model was generalized in several ways that can
be collected in the form
\begin{gather} \label{eq:kirchhoff2} \partial^2_{tt} u -M(\|u\|^2)
  \Delta u = f(t,x,u), \quad x \in \Omega
\end{gather}
for a suitable function $M \colon [0,\infty) \to \mathbb{R}$, called
\emph{Kirchhoff function}. The set~$\Omega$ is a bounded domain of
$\mathbb{R}^N$, and $\|u\|^2 = \|\nabla u\|_2^2$ denotes the Dirichlet
norm of $u$. The basic case corresponds to the choice
\begin{gather*}
M(t) = a + b t^{\gamma-1}, \quad a \geq 0, \ b \geq 0, \ \gamma \geq 1.
\end{gather*}
When $M(0)=0$, i.e. $a=0$, the equation is called
\emph{degenerate}. \emph{Stationary solutions} to
\eqref{eq:kirchhoff2} solve the equation
\begin{gather} \label{eq:kirchhoff2}
\begin{cases}
 -M(\|u\|^2) \Delta u  = f(x,u), & x \in \Omega \\
 u=0 &\hbox{on $\partial \Omega$}
 \end{cases}
\end{gather}
We refer to \cite{zbMATH07081235} for a recent survey of the results
connected to this model.

\bigskip

The existence and multiplicity of solutions to Kirchhoff problems
under the effect of a critical nonlinearity \(f\) have received
considerable attention. The term \emph{critical} refers here to the rough
assumption that $f(u) \sim |u|^{2^*-2} u$ with $2^* = 2N/(N-2)$. The
natural setting of the corresponding equation in $H_0^1(\Omega)$
yields a lack of compactness, since the embedding of
$H_0^1(\Omega)$ into $L^{2^*}(\Omega)$ is only
continuous. Straightforward techniques of Calculus of Variations thus fail,
and more advanced results from Critical Point Theory must be used. In
particular, P.-L.~Lions' Concentration-Compactness appears as a natural tool
for the analysis of the loss of compactness.

The relevant outcome is that the Kirchhoff function $M$ interacts with
the critical growth of the nonlinearity $g$: the validity of the
Palais-Smale compactness condition holds only under a condition like
\begin{gather*}
a^\frac{N-4}{2}b \geq C_2(N),
\end{gather*}
and a similar inequality ensures that the associated Euler functional
is weakly lower semicontinuous.

In the recent paper \cite{MR4201645}, Faraci and Silva
obtained several quantitative results for the problem
\begin{gather}
\label{eq:FS}
\begin{cases}
  - \displaystyle\left( a+b \int_\Omega |\nabla u|^2 \, dx \right) \Delta u = |u|^{2^*-2}u + \lambda g(x,u) &\hbox{in $\Omega$} \\
  u =0 &\hbox{on $\partial \Omega$,}
\end{cases}
\end{gather}
where $\Omega$ is an open bounded subset of $\mathbb{R}^N$, $N > 4$,
$a$ and $b$ are positive fixed numbers, $\lambda$ is a real parameter and
$g$ is a Carath\'{e}odory function that satisfies suitable growth
conditions. By using a fibering-type approach, the authors of
\cite{MR4201645} investigate existence, non-existence and multiplicity
of solutions to \eqref{eq:FS}. In the previous paper~\cite{faraci2018energy}, Faraci, Farkas and Kristály studied
equation~\eqref{eq:FS} with $g(x,u)=0$ and under suitable assumptions
on the parameters $a$ and $b$ they proved that the functional
associated to the problem is sequentially weakly lower semicontinuous,
satisfies the Palais-Smale condition and is convex.  \bigskip

The purpose of the present paper is to extend part of these results to
the \emph{fractional} counterpart of the Kirchhoff problem
\begin{equation} \label{eq:Pablambda}
\tag{$P_{a,b}^{\lambda}$}
\begin{cases}  \displaystyle\left( a+b \int_{\mathcal{Q}} \frac{|u(x)-u(y)|^2}{|x-y|^{n+2s}}\, dx\, dy \right)(-\Delta)^s u = |u|^{2^*_s-2}u+\lambda g(x,u) &\hbox{in}\ \Omega \\
  u=0 & \mbox{in} \ \R^N \setminus \Omega
\end{cases}
\end{equation}
where $\Omega\subset \R^N$ is a bounded domain with Lipschitz boundary
$\partial \Omega$, $\mathcal{Q}=\R^{2N}\setminus \mathcal{O}$ and
$\mathcal{O}=\Omega^c \times \Omega^c$, $a$ and $b$ are strictly
positive real numbers, $s \in (0,1)$, $N>4s$ and $2^*_s:=2N/(N-2s)$
denotes the critical exponent for the Sobolev embedding of $H^s(\R^N)$
into Lebesgue spaces. $g$ is a function that satisfies hypothesis
similar to the one in \eqref{eq:FS} adapted to the non local case. The
fractional Laplacian in \eqref{eq:Pablambda} is defined as
\begin{gather*}
  (-\Delta)^su(x) =K_{N,s} \, \lim_{\epsilon \to
    0^+}\int_{\R^N\setminus B_\epsilon(0)}
  \frac{u(x)-u(y)}{|x-y|^{N+2s}} \, dy
\end{gather*}
where
\begin{gather*}
  \frac{1}{K_{N,s}}:=\int_{\R^N}\frac{1-\cos\zeta_1}{|\zeta|^{N+2s}} \, d
    \zeta.
\end{gather*}
Since the parameter~$s$ is fixed, we will work with a rescaled version
of the operator and this enables us to assume that $K_{N,s}=1$. For
references about the fractional Laplacian we refer to \cite{MR2944369}, \cite{MR3967804} and to the monograph \cite{zbMATH06533015}. We define the space $X$ as the set of functions
$u \colon \R^N \to \R$ such that $u|_\Omega \in L^2(\Omega)$ and
\begin{gather*}
  \left\{ (x,y) \mapsto \frac{ u(x) -u(y)}{|x-y|^{N/2 +s}} \right\} \in L^2(\mathcal{Q}),
\end{gather*}
endowed with the norm
\begin{equation} \label{eq:1}
  \Vert u \Vert_X=\Vert u \Vert_{L^2(\Omega)} +
  \left(\int_{\mathcal{Q}} \frac{|u(x)-u(y)|^2}{|x-y|^{N+2s}} \, dx\,
    dy\right)^{\frac{1}{2}}.
\end{equation}
We also set
\[
X_0^s(\Omega):=\left\{u \in X:\ u=0 \ \mbox{a.e. in} \ \R^N \setminus \Omega \right\}.
\]
We introduce the best Sobolev constant for the continuous embedding $X_0^s(\Omega) \subset L^{2_s^*}(\Omega)$ as
\begin{equation} \label{Sobolev} S_{N,s}:=\inf_{u \in X_0^s(\Omega)}
  \frac{\Vert u \Vert^2}{\Vert u \Vert_{2^*_s}^2},
\end{equation}
where
\[
\Vert u \Vert^2:=\int_{\mathcal{Q}} \frac{\left| u(x)-u(y) \right|^2}{|x-y|^{N+2s}} \, dx\, dy.
\]
This norm is induced by the scalar
product
\[
  \langle u,v\rangle_{X_0^s(\Omega)}:=
  \int_{\mathcal{Q}}\frac{(u(x)-u(y))(v(x)-v(y))}{|x-y|^{N+2s}}\, dx\, dy
  \quad \mbox{for all} \, u,v \in X_0^s(\Omega)
\]
and we recall that in $X_0^s(\Omega)$ it is equivalent to
\eqref{eq:1}. For further details we refer the reader to \cite[Lemma
6]{MR2879266}. Weak solutions to
\eqref{eq:Pablambda} correspond to critical points of the
functional $\mathcal{I}_{a,b}^{\lambda} \colon X_0^s(\Omega) \to \R$
associated to the problem:
\[
  \mathcal{I}_{a,b}^{\lambda}(u):=\frac{a}{2}\Vert u \Vert^2+\frac{b}{4}\Vert u
  \Vert^{4}-\frac{1}{2^*_{s}}\Vert u \Vert^{2^*_s}_{2^*_s}- \lambda \int_{\Omega} G(x,u) \, dx,
\]
where we denote with $G(x,t)=\int_0^t g(x,\tau) d\tau$.
Arguing as in \cite[Proposition 1.12]{MR1400007}, we get
\begin{equation} \label{eq19}
  \left(\mathcal{I}_{a,b}^{\lambda}\right)'(u)\left[v\right]=\left( a+b\Vert u
    \Vert^2 \right) \langle
  u,v\rangle_{X_0^s(\Omega)}-\int_{\Omega}|u|^{2^*_s-2}uv\, dx
  -\lambda \int_{\Omega} g(x,u)v \, dx
\end{equation}
for all $u,v \in X_0^s(\Omega)$. When we have $g(x,u)=0$ we will use the notation
\[
  \mathcal{I}_{a,b}(u):=\frac{a}{2}\Vert u \Vert^2+\frac{b}{4}\Vert u
  \Vert^{4}-\frac{1}{2^*_{s}}\Vert u \Vert^{2^*_s}_{2^*_s}
\]
and we point out that $\mathcal{I}_{a,b}$ is a $C^2$-functional.

The interest in generalizing to the fractional case the model
introduced by Kirchhoff does not arise only for mathematical
purposes. In fact, following the ideas of \cite{MR2675483} and the
concept of fractional perimeter, Fiscella and Valdinoci proposed in
\cite{MR3120682} an equation describing the behavior of a string
constrained at the extrema in which appears the fractional length of
the rope. The interested reader can also consult \cite{MR2540182,MR3289358,MR2588587} and the references therein for
further motivations and applications of operators similar to the one
proposed in \eqref{eq:Pablambda}.

Recently, problems similar to \eqref{eq:1} have been extensively
investigated by many authors using different techniques and producing
several relevant results. In \cite{MR3120682} Fiscella and Valdinoci
showed the existence of a non-negative solution of mountain pass type
for an equation with a critical term perturbed with a subcritical
nonlinearity. In the same spirit, Autuori, Fiscella and Pucci generalize in \cite{MR3373607}
these results to the degenerate case, i.e $M(0)=0$, without any
monotonicity assumption on the function $M$. Iin these
two articles the operator taken into account is more general than the
one we consider here, but the two coincide making a particular choice
on the kernel; see also~\cite{MR3575909}. Liu, Squassina and
Zhang studied in~\cite{MR3679329} ground state solutions for the Kirchhoff equation plus a
potential with a non linear term asymptotic to a power with critical
growth in low dimension. It is also worth mentioning \cite{MR3961733}
where Mingqi, R\v{a}dulescu and Zhang proved the existence of
nontrivial radial solutions in the non-degenerate and degenerate cases
for the non local Kirchhoff problem in which the fractional Laplacian
is replaced by the fractional magnetic operator.

On the other hand, we could not find any overview of
different kind of solutions at different level of energy for the
fractional Kirchhoff problem. Although some of the results
we are going to prove are known, we present a proof based on a
adaptation to the fractional case due to Palatucci and Pisante
(\cite{MR3216834}) of the Lions second concentration-compactness
principle; for the original version of the lemma we refer to
\cite{MR850686}, as well as \cite{MR834360}.

\medskip

We collect here our main results.
\begin{theorem} \label{th1} Define
\[
 L_{N,s}:=\frac{4s(N-4s)^{\frac{N-4s}{2s}}}{N^{\frac{N-2s}{2s}}S_{N,s}^{\frac{N}{2s}}}, \quad \mathsf{PS}_{N,s}:=\frac{2s(N-4s)^{\frac{N-4s}{2s}}}{(N-2s)^{\frac{N-2s}{2s}}S_{N,s}^{\frac{N}{2s}}},
 \]
 and
 \[
 C_{N,s}:=\frac{2s(N-4s)^{\frac{N-4s}{2s}}(N+2s)^{\frac{N-2s}{2s}}}{(N-2s)^{\frac{N-2s}{s}}S_{N,s}^{\frac{N}{2s}}}.
\]
The following assertions holds true:
\begin{description}
\item[$(i)$] the energy functional $\mathcal{I}_{a,b}$ is sequentially weakly
  lower semicontinuous on $X_0^s(\Omega)$ if and only if
  $a^{\frac{N-4s}{2s}}b \geq L_{N,s}$.
\item[$(ii)$] If $a^{(N-4s)/2s}b\geq \mathsf{PS}_{N,s}$, the functional
$\mathcal{I}_{a,b}$ satisfies the compactness Palais-Smale condition at level
$c \in \R$.
\item[$(iii)$] If $a^{(N-4s)/2s}b\geq C_{N,s}$, then the functional $\mathcal{I}_{a,b}$ is convex on $X_0^s(\Omega)$.

\end{description}
\end{theorem}

Theorem \ref{th1} guarantees the validity of some crucial properties such as the sequentially weakly lower semicontinuity and the Palais-Smale condition. As we are going to see in the next statement, these facts enable us to use traditional variational methods to completely describe the situation for problem \eqref{eq:Pablambda}. We start with two results about the existence of global minimizers at different level of energy.

\begin{theorem} \label{th4} Let $a,b \in \R^+$ such that
  $a^{(N-4s)/2s}b \geq L_{N,s}$ and set
\begin{equation*}
  \iota_{\lambda}^s:=\inf \left\lbrace \mathcal{I}_{a,b}^{\lambda}(u)
    \mid \ u \in X_0^s (\Omega) \setminus \lbrace 0 \rbrace
  \right\rbrace \quad \mbox{for any} \ \lambda >0.
\end{equation*}
  There exists $\overline{\lambda}_0^s \geq 0$ such that for any $\lambda >\overline{\lambda}_0^s$ it is possible to find
  $u_{\lambda}^s \in X_0^s(\Omega) \setminus \lbrace 0 \rbrace$ such
  that
  $\mathcal{I}_{a,b}^{\lambda}(u_{\lambda}^s)=\iota^s_{\lambda}<0$.
\end{theorem}
\begin{theorem} \label{th6}
Let $\lambda=\overline{\lambda}^{s}_0$. The following statements hold:
\begin{description}
\item[$(i)$] if $a^{(N-2s)/2s}b>L_{N,s}$ then there exists
  $u_{\lambda}^s \in X_0^s(\Omega) \setminus \lbrace 0 \rbrace$ such
  that
  $\iota_{\overline{\lambda}_0^s}^s=\mathcal{I}_{a,b}^{\overline{\lambda}_0^s}=0$;
\item[$(ii)$] if $a^{(N-2s)/2s}b=L_{N,s}$, then $u=0$ in the only minimizer for $\iota_{\overline{\lambda}_0^s}^s$.
\end{description}
\end{theorem}
In the next Theorem we give some information on what happens when we do not keep fixed the parameters $a$, and $b$. It asserts we have some kind of stability when the product $a^{(N-4s)/2s}b$ becomes close to $L_{N,s}$.
\begin{theorem} \label{th5} Let $(a_k)_k$, $(b_k)_k$ be a sequence of
  real positive numbers such that $a_k \to a$, $b_k \to b$ and
  $a_k^{(N-4s)/2s}b_k \searrow L_{N,s}$. Setting
  $\lambda_k:=\overline{\lambda}_0^s(a_k,b_k)$ we have that
  $\lambda_k \to 0$ as $k \to \infty$. Furthermore, if
  $(u_k)_k \subset X_0^s(\Omega) \setminus \lbrace 0 \rbrace$ such
  that $\lambda_k=\lambda_0^s(u_k)$ then $u_k \rightharpoonup 0$ and
  \[\frac{\Vert u_k \Vert_{2^*_s}^2}{\Vert u_k \Vert^2} \to S_{N,s}.\]
\end{theorem}
Next statement shows the existence of solution of mountain pass type when $\lambda \geq \overline{\lambda}_0^s$.

\begin{theorem} \label{th7} If $\lambda \geq \overline{\lambda}_0^s$,
  then there exists a
  $v_{\lambda}^s \in X_0^s{\Omega} \setminus \lbrace 0 \rbrace$ such
  that $\mathcal{I}_{a,b}^{\lambda}(v_{\lambda}^s)=c_{\lambda}^s$ and
  $\left(\mathcal{I}_{a,b}^{\lambda}\right)' (v_{\lambda}^s)=0$ where
\[
c_{\lambda}^s:= \inf_{h \in \Gamma_{\lambda}^s} \max_{\zeta \in \left[ 0,1 \right]} \mathcal{I}_{a,b}^{\lambda}(h(\zeta))
\]
and
\[
\Gamma_{\lambda}^s:=\left\lbrace h \in C\left( \left[0,1\right], X_0^s(\Omega) \right)  \mid h(0)=0, \ h(1)=u_{\overline{\lambda}_0^s}^s \right\rbrace.
\]
\end{theorem}
Finally we focus on the case $\lambda \in (\overline{\lambda}_0^s-\delta, \overline{\lambda}_0^s)$, for some small~$\delta>0$.
\begin{theorem} \label{th8}
There exist $\delta>0$, $r>0$ such that for any
  $\lambda \in (\overline{\lambda}_0^s-\delta, \overline{\lambda}_0^s
  )$ the value 
  \[
\hat{\iota}_{\lambda}^s:= \inf \left\lbrace \mathcal{I}_{a,b}^{\lambda}(u) \mid u \in X_0^s(\Omega), \ \Vert u \Vert \geq r \right\rbrace
\]
is attained at a function
  $w_{\lambda}^s \in X_0^s(\Omega)$ satisfying
  $\Vert w_{\lambda}^s \Vert > r$.
\end{theorem}
\begin{theorem} \label{th9}
For any $\lambda \in (\overline{\lambda}_0^s-\delta, \overline{\lambda}_0^s )$ there is $v_{\lambda}^s \in X_0^s(\Omega) \setminus \lbrace 0 \rbrace$ such that $\mathcal{I}_{a,b}^{\lambda}(v_{\lambda}^s)=c_{\lambda}^s$ and $\left(\mathcal{I}_{a,b}^{\lambda}\right)'(v_{\lambda}^s)=0$, where
\begin{equation*}
c_{\lambda}^s:= \inf_{h \in \Gamma_{\lambda}^s} \max_{\zeta \in \left[ 0,1 \right]} \mathcal{I}_{a,b}^{\lambda} (h(\zeta))
\end{equation*}
and
\begin{equation*}
\Gamma_{\lambda}^s:= \left\lbrace h \in C\left( \left[0,1\right], X_0^s(\Omega) \right): h(0)=0, \ h(1)=w_{\lambda}^s \right\rbrace .
\end{equation*}
\end{theorem}

Our paper is organized as follows: in Section \ref{sec:1} we present the classic Kirchhoff model, its generalization to the non local case and we collect in a synthetic way our main results. In Section \ref{sec:2} we prove for the functional associated to the problem with $g(x,u)=0$ the
weak lower semicontinuity, the validity of the Palais-Smale condition
and the convexity under suitable assumption on the parameters $a$ and
$b$. Since the perturbation $g$ will have a subcritical growth, we
prove these conditions for the problem with the pure power
in order to ease notation. The general case requires 
only minor adjustments. In Section \ref{sec:3} we prove the existence
of global minimizers, local minimizers and mountain pass type
solutions with different energy level at varying of the parameter
$\lambda$. At the end of Section \ref{sec:3}, strengthening the
hypothesis on the non linear term $g$, we are able to give also a non
existence result for problem \eqref{eq:Pablambda}.

\section{Semicontinuity and the validity of the Palais-Smale
  condition} \label{sec:2}

In this section we completely describe the range of parameters $a$ and
$b$ for which the functional $\mathcal{I}_{a,b}$ associated to the
problem
\begin{equation} \label{eq:Pab}
\tag{$P_{a,b}$}
\begin{cases}  \displaystyle\left( a+b \int_{\mathcal{Q}} \frac{|u(x)-u(y)|^2}{|x-y|^{n+2s}}\, dx\, dy \right)(-\Delta)^s u = |u|^{2^*_s-2}u &\hbox{in}\ \Omega \\
  u=0 & \mbox{in} \ \R^N \setminus \Omega
\end{cases}
\end{equation}
is (sequentially) weakly lower semicontinuous.

\begin{proof}[Proof of Theorem \ref{th1} $(i)$]
  We assume that $a^{\frac{N-4s}{2s}}b \geq L_{N,s}$, and we
  choose a sequence $(u_n)_n \subset X_0^s(\Omega)$ such that
  $u_n \rightharpoonup u$. Since the embedding
  $X_0^s(\Omega) \hookrightarrow L^p(\Omega)$ is compact (see for
  instance \cite[Lemma 9]{MR3271254}), $u_n$ converges to $u$
  strongly in $L^p(\Omega)$ for any~$p \in \left[1,2^*_s \right)$. We
  notice that
  \begin{multline} \label{eq2} \Vert u_n-u \Vert^2+2\langle
    u_n-u,u\rangle_{X_0^s(\Omega)}=\langle
    u_n,u_n\rangle_{X_0^s(\Omega)}+\langle
    u,u\rangle_{X_0^s(\Omega)}-2\langle u_n,u\rangle_{X_0^s(\Omega)}
     \\ {} + 2 \langle u_n,u\rangle_{X_0^s(\Omega)} -2\langle
    u,u\rangle_{X_0^s(\Omega)}=\Vert u_n \Vert^2-\Vert u \Vert^2.
\end{multline}
Hence
\begin{gather*}
  \Vert u_n \Vert^2-\Vert u \Vert^2=\Vert u_n-u \Vert^2+2\langle
  u_n-u,u\rangle_{X_0^s(\Omega)}= \Vert u_n-u \Vert^2+o(1)
\end{gather*}
as $n \to \infty$.
After that, we compute
\begin{align} \label{eq3} \Vert u_n \Vert^4- \Vert u \Vert^4 &=\left(
    \Vert u_n \Vert^2- \Vert u \Vert^2 \right)\left( \Vert u_n
    \Vert^2+ \Vert u \Vert^2 \right)
                                                               \notag \\
                                                             &= \left(
                                                               \Vert
                                                               u_n-u
                                                               \Vert^2
                                                               +o(1)
                                                               \right)\left(
                                                               \Vert
                                                               u_n-u
                                                               \Vert^2
                                                               +
                                                               2\Vert
                                                               u
                                                               \Vert^2
                                                               + o(1)
                                                               \right).
\end{align}
Finally, using the Brezis-Lieb Lemma (see~\cite[Theorem 1]{MR699419}),
we have
\begin{equation} \label{eq4}
  \Vert u_n-u\Vert^{2^*_s}_{2^*_s}=\Vert
  u_n\Vert^{2^*_s}_{2^*_s}-\Vert u\Vert^{2^*_s}_{2^*_s} + o(1)
\end{equation}
as $n \to \infty$. Putting together \eqref{eq2}, \eqref{eq3},
\eqref{eq4} and the Sobolev inequality~\eqref{Sobolev} we obtain
\begin{align} \label{eq5}
  \mathcal{I}_{a,b}(u_n)-\mathcal{I}_{a,b}(u)&=\frac{a}{2}\left( \Vert
    u_n \Vert^2-\Vert u \Vert^2 \right)+\frac{b}{4}\left( \vert u_n
    \Vert^4-\Vert u \Vert^4 \right)-\frac{1}{2^*_s}\left( \Vert u_n
    \Vert^{2^*_s}_{2^*_s}-\Vert u \Vert^{2^*_s}_{2^*_s} \right)\notag
  \\ \notag &=\frac{a}{2}\Vert u_n-u \Vert^2+\frac{b}{4}\left(\Vert
    u_n-u \Vert^4+2 \Vert u \Vert^2 \Vert u_n-u \Vert^2 \right)\\
  \notag &-\frac{1}{2^*_s}\Vert u_n-u \Vert^{2^*_s}_{2^*_s}+o(1) \\
  \notag
                                             &\geq \frac{a}{2}\Vert u_n-u \Vert^2+\frac{b}{4}\Vert u_n-u \Vert^4-\frac{S_{N,s}^{-\frac{2^*_s}{2}}}{2^*_s}\Vert u_n-u \Vert^{2^*_s}+o(1)\\
                                             &= \Vert u_n-u \Vert^2
                                               \left[
                                               \frac{a}{2}+\frac{b}{4}\Vert
                                               u_n-u
                                               \Vert^2-\frac{S_{N,s}}{2^*_s}\Vert
                                               u_n-u \Vert^{2^*_s-2}
                                               \right] +o(1)
\end{align}
as $n \to \infty$. At this point, we introduce the auxiliary function
\begin{gather*}
  f_{N,s}(\zeta)=\frac{a}{2}+\frac{b}{4}\zeta^2-\frac{S_{N,s}^{-\frac{2^*_s}{2}}}{2^*_s}\zeta^{2^*_s-2},
  \quad \zeta \geq 0.
\end{gather*}
It is easy to verify that the function $f_{N,s}$ attains its minimum
at the point
\begin{gather*}
m_{N,s}=\left(\frac{b}{2} \frac{2^*_s }{2^*_s-2}S_{N,s}^{\frac{2^*_s}{2}}\right)^{\frac{1}{2^*_s-4}},
\end{gather*}
and that
\begin{equation} \label{eq6}
a^{\frac{N-4s}{2s}}b\geq L_{N,s} \Leftrightarrow f_{N,s}(m_{N,s})=\frac{1}{2}\left( a-b^{-\frac{2s}{N-4s}}L_{N,s}^{\frac{2s}{N-4s}}\right)\geq 0
\end{equation}
From \eqref{eq5} and \eqref{eq6} it follows that
\[
\liminf_{n \to \infty} \left( \mathcal{I}_{a,b}(u_n)-\mathcal{I}_{a,b}(u) \right) \geq \liminf_{n \to \infty} \Vert u_n - u \Vert^2 f_{N,s}(\Vert u_n-u \Vert) \geq 0,
\]
which concludes this part of the proof.

Conversely, we proceed by contradiction, assuming that the functional
$\mathcal{I}_{a,b}$ is sequentially weakly lower semicontinuous
but
\begin{equation} \label{eq7} a^{\frac{N-4s}{2s}}b<L_{N,s}
\end{equation}
Let~$\{u_n\}_n \subset X_0^s(\Omega)$ be a minimizing sequence for
$S_{N,s}$. By homogeneity we may assume furthermore that
$\|u_n\|_{2_s^*}=1$ for every $n$, so that we deduce that the sequence
$\{u_n\}_n$ must be bounded. Up to a subsequence, we have that
$u_n \rightharpoonup u $ in $X_0^s(\Omega)$ for some
$u \in X_0^s(\Omega) \setminus \lbrace 0\rbrace$. Besides, exploiting
the weak lower semicontinuity of the norm, we have that
$\Vert u \Vert \leq \liminf_{n \to \infty} \Vert u_n \Vert =:L$ and
there exists a subsequence $\{u_{n_k}\}_k$ such that
$L = \lim_{k \to \infty} \Vert u_{n_k} \Vert$. We point out that
$L>0$, since $u \neq 0$. Now $N>4s$ implies that $0<2^*_s-2<2$, and
$\lim_{x \to +\infty} f_{N,s}(x) =+\infty$.  As we have already seen,
the function $f_{N,s}$ attains its minimum at the point $m_{N,s}$ and
this, together with \eqref{eq7}, implies $f_{N,s}(m_{N,s})<0$. Set
$c=m_{N,s}/L>0$. We notice that
\begin{align} \label{eq8}
\liminf_{n \to \infty} \mathcal{I}_{a,b}(cu_n)& \leq \liminf_{k \to \infty} \mathcal{I}_{a,b}(c u_{n_k}) \notag \\ \notag
& = \liminf_{k \to \infty} \Vert c u_{n_k} \Vert^2f_{N,s}(\Vert c u_{n_k} \Vert)=(cL)^2f_{N,s}(cL)\\
& =(cL)^2 f_{N,s}(m_{N,s}) \leq \Vert cu \Vert^2 f_{N,s}(m_{N,s}) \leq \Vert cu \Vert^2 f_{N,s}(\Vert cu \Vert).
\end{align}
We also have that
\begin{align}
\notag \Vert cu \Vert^2 f_{N,s}(\Vert cu \Vert) &= \frac{a}{2} \Vert cu \Vert^2+\frac{b}{4}\Vert cu \Vert^4-\frac{S_{N,s}^{2^*_s/2}}{2^*_s}\Vert cu \Vert \\
& \leq \frac{a}{2} \Vert cu \Vert^2 + \frac{b}{4}\Vert cu \Vert^4-\frac{1}{2^*_s}\int_{\Omega}|cu|^{2^*_s} \, dx = \mathcal{I}_{a,b}(cu).
\end{align}
Comparing \eqref{eq7} with \eqref{eq8} we get
\begin{equation} \label{eq9}
\liminf_{n \to \infty} \mathcal{I}_{a,b}(c u_n) \leq \mathcal{I}_{a,b}(cu).
\end{equation}
We claim that a strict inequality holds in \eqref{eq9}. Indeed, if we
had equality, the function~$cu$ would attain the minimum in
\eqref{Sobolev}. This is impossible, since $\Omega \neq \R^N$
(see~\cite[Theorem 1.1]{MR2064421}). The proof is complete.
\end{proof}
%

\begin{proof}[Proof of Theorem \ref{th1} $(ii)$]
  Let $\{u_n\}_n \subset X_0^s(\Omega)$ be a $(PS)_c$ sequence,
  i.e. 
  \begin{equation*}
  \mathcal{I}_{a,b}(u_n) \to c, \quad
  \mathcal{I}_{a,b}'(u_n) \to 0
  \end{equation*}
  as $n \to +\infty$.  Recalling
  \eqref{Sobolev}, we observe that
\[
\mathcal{I}_{a,b}(u)=a \Vert u \Vert^2+b\Vert u \Vert^4-\int_{\Omega} |u|^{2^*_s}\, dx \geq a \Vert u \Vert^2+b\Vert u \Vert^4-S_{N,s}^{-\frac{2^*_s}{2}}\Vert u \Vert^{2^*_s}.
\]
Since $2^*_s<4$ we have that $\mathcal{I}_{a,b}$ is coercive, and from
that we can deduce the boundedness of the sequence $\{u_n\}_n$. From
\cite[Lemma 9]{MR3271254}, up to a subsequence, we have
\begin{equation*}
\begin{cases}
u_n \rightharpoonup u & \mbox{in} \, X_0^s(\Omega)\\
u_n \to u & \mbox{in $L^p(\Omega)$ for all $p \in \left[1, 2^*_s \right)$}  \\
u_n \to u & \mbox{a.e in} \, \R^N.
\end{cases}
\end{equation*}
Using the H\"older inequality, it is straightforward to see that the
sequence $\{u_n\}_n$ is also bounded in the space
$\mathcal{M}(\Omega)$,
thus there exists two finite
measures $\mu$ and $\nu$ such that
\[
  (-\Delta)^s u_n \rightharpoonup^* \mu \quad \mbox{and} \quad
  |u_n|^{2^*_s} \rightharpoonup^* \nu \quad \mbox{in} \,
  \mathcal{M}(\Omega)
\]
From \cite[Theorem 1.5]{MR3216834}, it follows that either
$u_n \to u $ in $L^{2^*_s}(\Omega)$ or there exist a set $J$ at most
countable, two real sequences $\{\mu_j\}_{j \in J}$, $\{\nu_j\}_{j \in J}$
and distinct points $\{x_j\}_{j\in J} \subset \R^N$ such that
\begin{equation} \label{eq10}
\nu=|u|^{2^*_s}+ \sum_{j \in J} \nu_j \delta_{x_j}
\end{equation}
and
\begin{equation} \label{eq11}
\mu= (-\Delta)^s u+\tilde{\mu}+\sum_{j \in J} \mu_j \delta_{x_j}
\end{equation}
for some positive finite measure $\tilde{\mu}$, where
\begin{equation} \label{eq12}
\nu_j \leq S_{N,s} \mu_j^{2^*_s}.
\end{equation}
\textit{Claim:} the set $J$ is empty.

If not, there exists an index $j_0$ such that $\nu_{j_0}\neq 0$ at
$x_{j_0}$. Fix $\varepsilon>0$ and consider a cut-off function
$\vartheta_{\varepsilon}$ such that
\begin{equation*}
\begin{cases}
0 \leq \vartheta_{\varepsilon} \leq 1 & \mbox{in} \, \Omega \\
\vartheta_{\varepsilon}=1 & \mbox{in} \, B(x_{j_0}, \varepsilon) \\
\vartheta_{\varepsilon}=0 & \mbox{in}\, \Omega \setminus B(x_{j_0}, 2\varepsilon).
\end{cases}
\end{equation*}
Since the sequence $\{u_n \vartheta_{\varepsilon}\}_n$ is still
bounded in $X_0^s(\Omega)$, we have that
\[
\lim_{n \to \infty} \mathcal{I}_{a,b}(u_n)\left[ u_n \vartheta_{\varepsilon}\right]=0,
\]
thus
\begin{align} \label{eq13}
\notag o(1)&=\mathcal{I}_{a,b}'(u_n)\left[u_n \vartheta_{\varepsilon}\right] =\left( a+b \Vert u_n \Vert^2 \right)\langle u_n,u_n \vartheta_{\varepsilon}\rangle_{X_0^s(\Omega)}-\int_{\Omega}|u_n|^{2^*_s}\vartheta_{\varepsilon}\, dx \\
\notag & =\left[ \left( a+b \Vert u_n \Vert^2 \right)  \int_{\mathcal{Q}} u_n(y)\frac{(u_n(x)-u_n(y))(\vartheta_{\varepsilon}(x)-\vartheta_{\varepsilon}(y))}{|x-y|^{N+2s}} \, dx\, dy \right.\\
& \left. +  \int_{\mathcal{Q}} \vartheta_{\varepsilon}(x) \frac{(u_n(x)-u_n(y))^2}{|x-y|^{N+2s}} \, dx\, dy \right] -\int_{\Omega}|u_n|^{2^*_s}\vartheta_{\varepsilon}\, dx .
\end{align}
as $n \to \infty$. By using the H\"older inequality, we estimate the
first term of \eqref{eq13}
\begin{align*}
 & \left( a+b \Vert u_n \Vert^2 \right)  \int_{\mathcal{Q}} u_n(y)\frac{(u_n(x)-u_n(y))(\vartheta_{\varepsilon}(x)-\vartheta_{\varepsilon}(y))}{|x-y|^{N+2s}} \, dx\, dy \\
&  \leq \int_{\mathcal{Q}} \frac{\left( u_n(x)-u_n(y)\right)^2}{|x-y|^{N+2s}}\, dx\, dy \int_{\mathcal{Q}} u_n^2(y)\frac{\left( \vartheta_{\varepsilon}(x)-\vartheta_{\varepsilon}(y)\right)^2}{|x-y|^{N+2s}}\, dx\, dy \\
& C\int_{\mathcal{Q}} u_n^2(y)\frac{\left( \vartheta_{\varepsilon}(x)-\vartheta_{\varepsilon}(y)\right)^2}{|x-y|^{N+2s}}\, dx\, dy
\end{align*}
for some $C>0$. As in~\cite[Lemma 2.1]{MR4018100}, we have that
\begin{equation} \label{eq14}
\lim_{\varepsilon \to 0} \limsup_{n \to \infty} \int_{\mathcal{Q}} |u_n(y)|^2 \frac{\left| \vartheta_{\varepsilon}(x)-\vartheta_{\varepsilon}(y)\right|^2}{|x-y|^{N+2s}}\, dx\, dy=0.
\end{equation}
Regarding the second term of \eqref{eq13}, recalling \eqref{eq11}, we
get
\begin{multline*}
 \lim_{n \to \infty} \left( a+b \Vert u_n \Vert^2 \right)  \int_{\mathcal{Q}} \vartheta_{\varepsilon}(x) \frac{(u_n(x)-u_n(y))^2}{|x-y|^{N+2s}} \, dx\, dy \\
\geq \lim_{n \to \infty}\left[  a \int_{\R^{2N}\setminus B(x_{j_0},2\varepsilon)^c\times \Omega^c} \vartheta_{\varepsilon}(x) \frac{(u_n(x)-u_n(y))^2}{|x-y|^{N+2s}} \, dx\, dy \right. \\
 \left. \quad + b \left( \int_{\mathcal{Q}} \vartheta_{\varepsilon}(x) \frac{(u_n(x)-u_n(y))^2}{|x-y|^{N+2s}} \, dx\, dy \right)^2 \right] \\
  \geq a \int_{\R^{2N}\setminus B(x_{j_0},2\varepsilon)^c\times \Omega^c} \vartheta_{\varepsilon}(x) \frac{(u(x)-u(y))^2}{|x-y|^{N+2s}} \, dx\, dy +a \mu_{j_{0}} \\
 +  b \left( \int_{\mathcal{Q}} \vartheta_{\varepsilon}(x) \frac{(u(x)-u(y))^2}{|x-y|^{N+2s}} \, dx\, dy \right)^2+b \mu_{j_{0}}^2.
\end{multline*}
Hence
\begin{equation} \label{eq15}
\lim_{\varepsilon \to 0} \lim_{n \to \infty} \left( a+b \Vert u_n \Vert^2 \right)  \int_{\mathcal{Q}} \vartheta_{\varepsilon}(x) \frac{(u_n(x)-u_n(y))^2}{|x-y|^{N+2s}} \, dx\, dy \geq a \mu_{j_0}+b\mu_{j_0}^2.
\end{equation}
Finally, exploiting \eqref{eq10} we have
\begin{equation} \label{eq16}
\lim_{\varepsilon \to 0} \lim_{n \to \infty} \int_{\Omega}|u_n|^{2^*_s} \vartheta_{\varepsilon}\, dx= \lim_{\varepsilon \to 0}\int_{\Omega}|u|^{2^*_s} \vartheta_{\varepsilon}\, dx+ \nu_{j_0}=\nu_{j_0}.
\end{equation}
Putting together \eqref{eq14}, \eqref{eq15} and \eqref{eq16}, and
using \eqref{eq12}, we obtain
\begin{gather*}
0\geq a \mu_{j_0}+b \mu_{j_0}^2-\nu_{j_0} \geq  a \mu_{j_0}+b \mu_{j_0}^2 - S_{N,s}^{-\frac{2^*_s}{2}}\mu_{j_0}^{\frac{2^*_s}{2}}
 =\mu_{j_0}\left(a+b \mu_{j_0}-S_{N,s}^{-\frac{2^*_s}{2}}\mu_{j_0}^{\frac{2^*_s}{2}-1} \right).
\end{gather*}
We define
\[
\tilde{f}_{N,s}(\zeta)= a+b\zeta-S_{N,s}^{-\frac{2^*_s}{2}}\zeta^{\frac{2^*_s}{2}-1} \quad \mbox{for} \, \zeta \geq 0.
\]
At this point, noting that the condition
$a^{(N-4s)/2s}b>\mathsf{PS}_{N,s}$ implies $\tilde{f}_{N,s}(x)>0$, we
deduce
\[
a+b\mu_{j_0}-S_{N,s}^{-\frac{2^*_s}{2}}\mu_{j_0}^{\frac{2^*_s}{2}-1}>0.
\]
Hence $\mu_{j_0}=0$, and recalling \eqref{eq12} $\nu_{j_0}=0$ as well.

So the set $J=\emptyset$, and using the Brezis-Lieb lemma (see
\cite[Theorem 1]{MR699419}) we can rewrite \eqref{eq10} as
\[
\lim_{n \to \infty} \int_{\Omega} |u_n|^{2^*_s}\, dx= \int_{\Omega}|u|^{2^*_s}\, dx.
\]
Hence $u_n \to u$ in $L^{2^*_s}(\Omega)$ and
\begin{equation} \label{eq17}
\lim_{n \to \infty }\int_{\Omega}|u_n|^{2^*_s-2}u_n(u-u_n) \, dx=0.
\end{equation}
Coupling $\eqref{eq17}$ and the fact that
$\mathcal{I}_{a,b}'(u_n) \to 0$ as $ n \to \infty$ we get
\begin{align*}
0&=\lim_{n \to \infty} \mathcal{I}_{a,b}'(u_n)\left[u_n-u \right]=\lim_{n \to \infty} \left[\left(a+b\Vert u_n \Vert^2\right) \langle u_n, u_n-u\rangle_{X_0^s(\Omega)} \right.\\
&\left. -\int_{\Omega}|u_n|^{2^*_s-2}u_n(u_n-u)\, dx \right]\\
&=\lim_{n \to \infty} \left(a+b\Vert u_n \Vert^2\right) \langle u_n, u_n-u\rangle_{X_0^s(\Omega)}.
\end{align*}
From the last chain of equalities, recalling that
$\{u_n\}_n \subset X_0^s(\Omega)$ is bounded, we obtain
\begin{equation} \label{eq18}
\lim_{n \to \infty} \langle u_n, u_n-u\rangle_{X_0^s(\Omega)}=0.
\end{equation}
To conclude the proof it suffices to notice that thanks to
\eqref{eq18} and $u_n \rightharpoonup u$ we have
\[
\Vert u_n -u\Vert^2=\langle u_n, u_n-u\rangle_{X_0^s(\Omega)}-\langle u, u_n-u\rangle_{X_0^s(\Omega)} \to 0
\]
as $n \to \infty$.
\end{proof}

\begin{proof}[Proof of Theorem \ref{th1} $(iii)$]
  In order to establish the convexity we will show that
\[
\mathcal{I}''_{a,b}(u)\left[v,v\right] \geq 0 \quad \mbox{for all $u$, $v \in X_0^s(\Omega)$}.
\]
Differentiating \eqref{eq19} we notice that
\begin{equation} \label{eq20}
\mathcal{I}''_{a,b}(u)\left[v,v\right]=a\Vert v \Vert^2+b\Vert u \Vert^2 \Vert v \Vert^2-(2^*_s-1)\int_{\Omega} |u|^{2^*_s-2}v^2\, dx.
\end{equation}
Using the H\"older and the Sobolev inequalities we get
\begin{equation} \label{eq21}
\int_{\Omega}|u|^{2^*_s-2}v^2\, dx \leq \Vert u \Vert_{2^*_s}^{2^*_s-2}\Vert v\Vert_{2^*_s}^2 \leq S_{N,s}^{-\frac{2^*_s}{2}} \Vert u \Vert^{2^*_s-2}\Vert v \Vert^2.
\end{equation}
Putting together \eqref{eq20} and \eqref{eq21} we obtain
\[
\mathcal{I}''_{a,b}(u)\left[v,v\right] \geq \Vert v \Vert^2 \left[ a+b\Vert u \Vert^2-(2^*_s-1)S_{N,s}^{-\frac{2^*_s}{2}}\Vert u \Vert^{2^*_s-2} \right].
\]
At this point we set
\[
\hat{f}_{N,s}(\zeta)=a+b\zeta^2-(2^*_s-1)S_{N,s}^{\frac{2^*_s}{2}}\zeta^{2^*_s-2} \quad \mbox{for all} \, \zeta \geq 0,
\]
and we want to prove that it is positive on $\left[0,\infty
\right)$. Indeed, with a simple computation it is possible to show
that $\hat{f}_{N,s}$ attains its global minimum at
\[
\hat{m}_{N,s}=\left( \frac{2b S_{N,s}^{\frac{2^*_s}{2}}}{(2^*_s-1)(2^*_s-2)} \right)^{\frac{1}{2^*_s-4}}
\]
and that
\[
\hat{f}_{N,s}(\zeta)\geq 0 \Leftrightarrow a^{\frac{N-4s}{2s}}b\geq C_{N,s}
\]
for all $\zeta\geq 0$.
\end{proof}
\begin{remark} \label{rem1} It is clear from the proof that the
  functional $\mathcal{I}_{a,b}$ is strictly convex provided that
  $a^{(N-4s)/2s}b> C_{N,s}$.
\end{remark}

\section{Application to a perturbed Kirchhoff problem} \label{sec:3}

This section is devoted to study an application of Theorem
$\ref{th1}$. More precisely we want to study the set of solutions of
the perturbed problem
\begin{equation} \label{eq:Pablambda}
\tag{$P_{a,b}^{\lambda}$}
\begin{cases}  \displaystyle\left( a+b \int_{\mathcal{Q}} \frac{|u(x)-u(y)|^2}{|x-y|^{n+2s}}\, dx\, dy \right)(-\Delta)^s u = |u|^{2^*_s-2}u+\lambda g(x,u) &\hbox{in}\ \Omega \\
  u=0 & \mbox{in} \ \R^N \setminus \Omega
\end{cases}
\end{equation}
where as before $a$, $b$ are real positive parameter, $\Omega$ is a
bounded domain and $\lambda >0$. As for $g$, we generalize to
the fractional case the assumptions present in
\cite{MR4201645}. Namely, we make the following assumptions:
\begin{itemize}
\item[($H_1$)] $g\colon \Omega \times \R \to \R$ is a Carath\'eodory
  function such that $g(x,0)=0$ a.e. in $\Omega$;
\item[($H_2$)] $g(x,t)>0$ for every $t>0$ and $g(x,t)<0$ for every
  $t<0$ a.e. in $\Omega$. In addition, we require that there is a
  $\mu >0$ such that $g(x,t)\geq \mu >0$ a.e in $\Omega$ and for every
  $t \in I$, where $I$ is some open interval of $(0,\infty)$;
\item[($H_3$)] there is a constant $c>0$ and $p \in (2,2^*_s)$ such
  that $g(x,t) \leq c(1+|t|^{p-1})$ a.e. in $\Omega$;
\item[($H_4$)] $\lim_{t \to 0} g(x,t)/|t|=0$ uniformly with respect to
  $x \in \Omega$.
\end{itemize}

Using a variational approach, we investigate the existence of critical
points of the functional defined on the space $X_0^s(\Omega)$
\[
  \mathcal{I}_{a,b}^{\lambda}(u):=\frac{a}{2}\Vert u \Vert^2+\frac{b}{4}\Vert u
  \Vert^{4}-\frac{1}{2^*_{s}}\Vert u \Vert^{2^*_s}_{2^*_s}-\lambda \int_{\Omega} G(x,u) \, dx
\]
where we denote with $G(x,t)=\displaystyle\int_0^t g(x,\tau) d\tau$.

We begin the treatment of our problem by proving a series of technical
results that will be useful throughout this section.
\begin{remark} \label{rem2}
Before starting, let us recall the functions
\[
f_{N,s}(\zeta):=\frac{a}{2}+\frac{b}{4}\zeta^2-\frac{S_{N,s}^{-\frac{2^*_s}{2}}}{2^*_s}\zeta^{2^*_s-2}
\]
and
\[
\tilde{f}_{N,s}(\zeta)= a+b\zeta-S_{N,s}^{-\frac{2^*_s}{2}}\zeta^{\frac{2^*_s}{2}-1}
\]
defined in the proofs of Theorems \ref{th1} $(i)$ and \ref{th1}
$(ii)$. As we have already seen these functions have a unique local
minimizer attained respectively at
\begin{gather*}
m_{N,s}=\left[\frac{b}{2} \frac{2^*_s }{2^*_s-2}S_{N,s}^{\frac{2^*_s}{2}}\right]^{\frac{1}{2^*_s-4}},
\end{gather*}
and
\begin{gather*}
\tilde{m}_{N,s}=\left[ \frac{2b}{2^*_s-2}S_{N,s}^{\frac{2^*_s}{2}}\right]^{\frac{1}{2^*_s-4}}.
\end{gather*}
Furthermore, $f_{N,s}(m_{N,s})>0$ if and only if
$a^{\frac{N-4s}{2s}}b >L_{N,s}$ and $f_{N,s}(m_{N,s})=0$ when
$a^{\frac{N-4s}{2s}}b =L_{N,s}$. Analogously
$\tilde{f}_{N,s}(\tilde{m}_{N,s})>0$ if and only if
$a^{(N-4s)/2s}b > \mathsf{PS}_{N,s}$ and
$\tilde{f}_{N,s}(\tilde{m}_{N,s})=0$ when
$a^{(N-4s)/2s}b = \mathsf{PS}_{N,s}$.
\end{remark}
\begin{proposition} \label{prop1} Let
  $u \in X_0^s(\Omega) \setminus \lbrace0\rbrace$. We have that:
\begin{description}
\item[$(i)$] for every $\zeta>0$ it holds
\[
\frac{a}{2}\Vert u \Vert^2+\frac{b}{4}\zeta^2\Vert u \Vert^4-\frac{1}{2^*_s}\zeta^{2^*_s-2}>f_{N,s}(\zeta \Vert u \Vert)\Vert u \Vert^2;
\]
\item[$(ii)$] for every $\zeta >0$ it holds
\[
a \Vert u \Vert^2 +b \zeta^2 \Vert u \Vert^4-\Vert u \Vert_{2^*_s}^{2^*_s}\zeta^{2^*_s-2}>\tilde{f}_{N,s}(\zeta \Vert u \Vert) \Vert u \Vert^2.
\]
\end{description}
\end{proposition}
\begin{proof}
  From the boundedness of \(\Omega\) it follows as in~\cite{MR2064421}
  that
\begin{align} \label{eq22}
 \notag \zeta^2 \left[ \frac{a}{2} \Vert u \Vert^2+ \frac{b}{4}\zeta^2\Vert u \Vert^4-\frac{1}{2^*_s}\zeta^{2^*_s-2}\Vert u \Vert^{2^*_s}_{2^*_s} \right] &= \frac{a}{2} \left( \zeta \Vert u \Vert\right)^2 + \frac{b}{4}\left( \zeta \Vert u \Vert\right)^4-\frac{\Vert u \Vert^{2^*_s}_{2^*_s}}{\Vert u \Vert^{2^*_s}} \frac{\left( \zeta \Vert u \Vert\right)^{2^*_s}}{2^*_s} \\
&>\frac{a}{2} \left( \zeta \Vert u \Vert\right)^2 + \frac{b}{4}\left( \zeta \Vert u \Vert\right)^4-S_{N,s}^{-\frac{2^*_s}{2}} \frac{\left( \zeta \Vert u \Vert\right)^{2^*_s}}{2^*_s}
\end{align}
where in the last expression we used the Sobolev inequality. Dividing
by $\zeta^2$ we get the first statement. $(ii)$ follows similarly.
\end{proof}
As we did in the previous section, we show in the following lemma that
the functional $\mathcal{I}_{a,b}^{\lambda}$ is sequentially lower
semicontinuous and satisfies the Palais-Smale condition for $a$ and
$b$ sufficiently large.
\begin{lemma} \label{lemma1} Let $a, b \in \R^+$,
  $(u_k)_k \subset X_0^s(\Omega)$ and $\lambda_k \to \lambda \geq 0$
  as $k \to \infty$:
\begin{description}
\item[$(1)$] if $a^{(N-4s)/2s}b\geq L_{N,s}$  and $u_k \rightharpoonup u$ in $X_0^s (\Omega)$ then $$\mathcal{I}_{a,b}^{\lambda} (u) \leq \liminf_{k \to \infty}\mathcal{I}_{a,b}^{\lambda}(u_k);$$
\item[$(2)$] if $a^{(N-4s)/2s}b > \mathsf{PS}_{N,s}$, $\mathcal{I}_{a,b}^{\lambda}(u_k) \to c$ and $\left(\mathcal{I}^{\lambda}_{a,b}\right)'(u_k) \to 0$ then $(u_k)_k$ is convergent to some $u$ in $X_0^s(\Omega)$ up to subsequence.
\end{description}
\end{lemma}
\begin{proof}
  The proof follows closely the arguments of Theorem \ref{th1} $(i)$
  and $(ii)$ with minor changes.
\end{proof}
Now choose $\lambda \geq 0$ and $u \in X_0^s(\Omega)$. For every
$\zeta>0$ we introduce the fiber map
\[
\mathcal{J}_{a,b}^{\lambda,u}(\zeta):=\mathcal{I}_{a,b}^{\lambda}(\zeta u)= \frac{a}{2} \zeta^2 \Vert u \Vert^2+ \frac{b}{4}\zeta^4\Vert u \Vert^4-\frac{\zeta^{2^*_s}}{2^*_s}\Vert u \Vert^{2^*_s}_{2^*_s}-\int_{\Omega}G(x,\zeta u)\, dx.
\]
\begin{proposition} \label{prop2} Let $\lambda \in \R$ be nonnegative
  and $u \in X_0^s(\Omega \setminus \lbrace0\rbrace$. Then there
  exists a neighbourhood $V_{\lambda}$ of $0$ such that
  $\mathcal{J}_{a,b}^{\lambda,u}(\zeta)>0$ for every
  $\zeta \in V_{\lambda} \cap (0, \infty)$. We also have that
  $\mathcal{J}_{a,b}^{\lambda,u}(\zeta)\to \infty$ as
  $ \zeta \to \infty$.
\end{proposition}
\begin{remark} \label{rem3} The previous proposition shows indirectly
  that the map $\mathcal{J}_{a,b}^{\lambda,u}(\zeta)$ is bounded from
  below
\end{remark}
\begin{proof}
  Fix $\varepsilon>0$. Exploiting $(H_4)$, for $\zeta$ small enough
  we get
\begin{align*}
\mathcal{J}_{a,b}^{\lambda,u}(\zeta)& = \zeta^2 \left(\frac{a}{2} \Vert u \Vert^2+ \frac{b}{4} \zeta^2\Vert u \Vert^4-\frac{\zeta^{2^*_s-2}}{2^*_s}\Vert u \Vert_{2^*_s}^{2^*_s}-\lambda \int_{\Omega} \frac{G(x,\zeta u)}{\zeta^2} \, dx \right)\\
&\geq \zeta^2 \left(\frac{a}{2} \Vert u \Vert^2+ \frac{b}{4}\zeta^2\Vert u \Vert^4-\frac{\zeta^{2^*_s-2}}{2^*_s}\Vert u \Vert_{2^*_s}^{2^*_s}-\lambda \frac{\varepsilon}{2} \Vert u \Vert_2^2 \right).
\end{align*}
Using the Sobolev inequality, taking $\varepsilon$ appropriately and
choosing $\zeta$ even smaller if necessary we obtain the first part of
the statement. In order to complete the proof, it is sufficient to
remember that $G$ has subcritical growth and to notice that
$2<p<2^*_s<4$.
\end{proof}
Now we choose $u \in X_0^s(\Omega)$ and we consider the system
\begin{equation} \label{eq23}
\begin{cases}
\mathcal{J}_{a,b}^{\lambda,u}(\zeta)=0 \\
(\mathcal{J}_{a,b}^{\lambda,u})'(\zeta)=0 \\
\mathcal{J}_{a,b}^{\lambda,u}(\zeta)=\inf_{\varrho>0} \mathcal{J}_{a,b}^{\lambda,u}(\varrho)
\end{cases}
\end{equation}
in the unknowns $\lambda$ and $\zeta$.
\begin{proposition} \label{prop:3.6}
Let $T$ and $Z$ be two topological space, and assume that $Z$ is compact. Let
  $h\colon T \times Z \to \R$ be a continuous function. Then the function $\hat{h}(t):=\inf_{z \in Z} h(t, z)$ is continuous on $T$.
\end{proposition}
\begin{proof}
We first observe that for any $t \in T$ the
  function $\hat{h}$ is well defined since $Z$ is compact and the
  infimum is always attained at some point $z(t) \in Z$. Recalling
  that the sets $(-\infty,a)$ and $(b, \infty)$ for some $a, b \in \R$
  form a subbase of $\R$, our proof is reduced to the following:

  \noindent \textit{Claim}: $\hat{h}^{-1} \left( -\infty,a\right)$ and
  $ \hat{h}^{-1}\left(b,\infty \right)$ are open in $T$.

  \noindent We start showing the truthfulness of the claim for
  $(-\infty,a)$. Denote with $\pi_{T}:T \times Z \to T$ the usual
  projection and remember that is a map continuous and open. Noticing
  that
  $\hat{h}^{-1}(-\infty,a)=\left(\pi_{T} \circ h^{-1}
  \right)(-\infty,a)$ it is straightforward to conclude. On the other
  hand, consider an half line $(b,\infty)$ for some $ b \in \R$. If
  $ t \in T$ is such that $\hat{h}(t) >b$ then $h(t,z) >b$ for any
  $z \in Z$. In other words if $t \in \hat{h}^{-1}(b,\infty)$ then
  $(t,z) \in h^{-1}(b,\infty)$ for any $z \in Z$. Since $h$ is
  continuous and $(b,\infty)$ is open, for any
  $(t,z) \in \hat{h}^{-1}(b,\infty) \times Z$ it is possible to find a
  neighbourhood $U_{t,z} \times V_{t,z}$ such that
\[
  (t,z) \in U_{t,z} \times V_{t,z} \subset h^{-1}(b,\infty).
\]
Hence $\left\{V_{t,z}\right\}_{z \in Z}$ is an open covering of $Z$. Exploiting the compactness of $Z$, we can extract a finite
subcovering indexed by a finite set $\mathcal{K}(t)$ with the property
\[
\lbrace t \rbrace \times Z \subset \bigcap_{i\in \mathcal{K}(t)} U_{t,z_i} \times Z \subset h^{-1}(b,\infty).
\]
Thus, we can conclude observing that
\begin{equation} \label{eq35}
\hat{h}^{-1}\left(b, \infty \right)= \bigcup_{t\in \hat{h}^{-1}(b,\infty)} \bigcap_{i \in \mathcal{K}(t)} U_{(t,z_i)}.
\end{equation}
\end{proof}
\begin{remark}
We can even strengthen the result above for functions defined on non
compact spaces requiring divergence at infinity. For instance,
suppose $h\colon (\R_+)^2 \to \R$ is continuous and such that
$\lim_{z \to \infty} h(t,z)=\infty$ for any $t \in \R_+$. The proof
for sets as $(-\infty, a)$ is the same. As regard sets of the type
$(b, \infty)$ we observe that $\hat{h}^{-1}(b,\infty)$ can be written
as in \eqref{eq35} plus an half line due to the divergence of the
function at infinity.
\end{remark}
\begin{proposition} \label{prop3}
Let $a, b \in \R^+$ such that $a^{(N-4s)/2s}b\geq L_{N,s}$. For any $u \in X_0^s(\Omega) \setminus \lbrace 0\rbrace$ there is a unique $\lambda = \lambda_0^s(u)$ that solves \eqref{eq23}.
\end{proposition}
\begin{proof}
  Define the continuous function
  $h(\lambda, \zeta):=\mathcal{J}_{a,b}^{\lambda,u}(\zeta)$. We start
  pointing out that $h(0,\zeta)$ is positive on $(0,\infty)$ (see
  Remark \ref{rem2}) and goes to $+\infty$ as $\zeta \to \infty$. By
  continuity we have that for $\lambda$ small $h(\lambda, \zeta)$ is
  nonnegative for all $\zeta \in \R^+$. Moreover, from Proposition
  \ref{prop2} it follows that for any $\lambda \geq 0$ there is a
  neighbourhood $V_{\lambda}$ such that $h(\lambda, \zeta)>0$ for all
  $\zeta \in V_{\lambda} \cap (0,\infty)$. We also have that
  $h(\lambda, \zeta) \to - \infty$ as $\lambda \to \infty$ for any
  $\zeta >0$. At this point we define the continuous function (refer
  to Proposition \ref{prop:3.6})
  $i(\lambda)=\inf_{\zeta \in \left[0,\infty \right)} h(\lambda,
  \zeta)$. From the previous considerations, we can deduce that for
  $\lambda$ sufficiently large the function $i$ is negative, while if
  we restrict $\lambda$ it is equal to zero. This is due to the fact
  that the function $h(\lambda, \zeta)$ for $\lambda$ big enough has a
  global minimizer in the variable $\zeta$ at a negative
  level. Shrinking $\lambda$, and remembering that all continuous
  functions are homotopically equivalent, this minimizer becomes local
  and attained at a positive level. All these arguments ensure us the
  existence of the desired $\lambda_{0}^s(u)$ that solves
  \eqref{eq23}.
\end{proof}
\begin{corollary} \label{cor1} Let
  $u \in X_0^s(\Omega)\setminus \lbrace0\rbrace$. The number
  $\lambda_0^s(u)$ is the only parameter such that
  \begin{equation*}
  \inf_{\zeta \in (0,\infty)}
  \mathcal{J}_{a,b}^{\lambda_0^s(u),u}(\zeta)=0.
  \end{equation*} 
  In addition,
\begin{displaymath}
\inf_{\zeta \in (0,\infty)} \mathcal{J}_{a,b}^{\lambda,u}(\zeta)
\begin{cases}
<0 &\hbox{if $\lambda > \lambda_0^s(u)$} \\
=0 &\hbox{if $0 \leq \lambda \leq \lambda_0^s(u)$}.
\end{cases}
\end{displaymath}
\end{corollary}
\begin{proof}
  The statement follows immediately from the proof of Proposition
  \ref{prop3}.
\end{proof}
Now we define a suitable parameter independent from $u$ that will play
a crucial r\^{o}le in the sequel. More precisely, we set
\[
  \overline{\lambda}_0^{s}:= \inf_{u \in X_0^s(\Omega) \setminus
    \lbrace 0\rbrace} \lambda_0^s(u).
\]
The next Proposition shows how the parameter $\overline{\lambda}_0^{s}$
varies depending on the choice made on $a$ and $b$.
\begin{proposition} \label{prop6}
The following statements hold:
\begin{description}
\item[$(i)$] if $a^{(N-4s)/2s}b>L_{N,s}$ then
  $\overline{\lambda}_0^s>0;$
\item[$(ii)$] if $a^{(N-4s)/2s}b=L_{N,s}$ then
  $\overline{\lambda}_0^s=0$. Furthermore, if
  $(u_k)_k \subset X_0^s(\Omega) \setminus \lbrace 0 \rbrace$ is a
  sequence such that $\lambda_0^s(u_k) \to \overline{\lambda}_0^s$ as
  $k \to \infty$, we have that $u_k \rightharpoonup 0$ and
  \[\frac{\Vert u_k \Vert_2^2}{\Vert u_k \Vert_{2^*_s}^2} \to S_{N,s}.\]
\end{description}
\end{proposition}
\begin{proof}
  $(i)$ As a first step we observe that the function
  $u \to \lambda_0^s(u)$ is well defined and homogeneous of degree
  zero. In fact, taking a couple $(\zeta, \lambda_0^s(u))$ that solves
  \eqref{eq23} and $\mu >0$, since
  $\mathcal{J}_{a,b}^{\lambda, \mu
    u}(\zeta)=\mathcal{J}_{a,b}^{\lambda, u}(\mu \zeta)$ and
  $\left(\mathcal{J}_{a,b}^{\lambda, \mu
    u}\right)'(\zeta)=\left(\mathcal{J}_{a,b}^{\lambda, u}\right)'(\mu \zeta)$ we have that
  also $(\frac{\zeta}{\mu},\lambda_0^s)$ is a solution of
  \eqref{eq23}. From the uniqueness of the parameter
  $\lambda^s_0(\mu u)$ it follows that
  $\lambda_0^s(\mu u)=\lambda_0^s(u)$. Now, assume by contradiction
  that $\overline{\lambda}_0^s=0$. If that, there is a sequence
  $(u_k)_k \subset X_0^s(\Omega) \setminus \lbrace 0 \rbrace$ such
  that $\lambda_k:=\lambda_0^s(u_k) \to 0$. By homogeneity, we may assume that $\Vert u_k \Vert=1$. From
  Proposition \ref{prop3} it follows that there exists $\zeta_k>0$
  such that $\mathcal{J}_{a,b}^{\lambda_k,u_k}(\zeta_k)=0$, that is
\[
  \frac{a}{2}+\frac{b}{2}\zeta_k^2-\frac{1}{2^*_s}\Vert u_k
  \Vert_{2^*_s}^{2^*_s} \zeta_k^{2^*_s-2}-\lambda_k \int_{\Omega}
  \frac{G(x,\zeta_k u)}{\zeta_k^2}\, dx=0.
\]
Recalling Remark \ref{rem2}, we get that
\begin{equation} \label{eq24} f_{N,s}(\zeta_k)<
  \frac{a}{2}+\frac{b}{2}\zeta_k^2-\frac{1}{2^*_s}\Vert u_k
  \Vert_{2^*_s}^{2^*_s} \zeta_k^{2^*_s-2}=\lambda_k \int_{\Omega}
  \frac{G(x,\zeta_k u)}{\zeta_k^2}\, dx.
\end{equation}
Hypotheses $H_3$ and $H_4$ implies that for any $\varepsilon>0$ there
exists a positive constant $c>0$ such that
$|G(x,t)|<\frac{\varepsilon}{2}t^2+\frac{c}{p}|t|^p$ for all
$x \in \Omega$ and all $t \in \R$. So, the sequence $(\zeta_k)_k$ must
be bounded, and up to subsequence converges to some
$\overline{\zeta}>0$. At this point, letting $k \to \infty$,
\eqref{eq24} becomes
\[
0<f_{N,s}(\overline{\zeta})= \lim_{k \to \infty} \lambda_k \int_{\Omega} \frac{G(x, \zeta_k u_k)}{\zeta_k^2} \, dx=0
\]
which is clearly a contradiction.

$(ii)$ Up to a translation, we can suppose that $0 \in \Omega$. Take a
nonnegative cut-off function such that $\varphi (x)=1$ in $B_R(0)$ for
some $R>0$. Fix $\varepsilon>0$ and consider
\[
v_{\varepsilon}(x):=\frac{\varphi(x)}{\left( \varepsilon + |x|^2 \right)^{\frac{N-2s}{2}}}.
\]
We set $u_{\varepsilon}:= v_{\varepsilon}/ \Vert v_{\varepsilon} \Vert$ and we notice that from \cite[Propositions 21 and 22]{MR3271254} it follows that
\[
\Vert u_{\varepsilon}\Vert=1, \quad \Vert u_{\varepsilon}\Vert_{2^*_s}^{2^*_s} \geq S_{N,s}^{-\frac{2^*_s}{2}}+O(\varepsilon^{\frac{N-2s}{2}}), \quad \Vert v_{\varepsilon} \Vert \leq \varepsilon^{-\frac{N-2s}{4}} C_1+O(1)
\]
as $\varepsilon \to 0$ for some $C_1>0$. In virtue of the previous estimates, we get
\begin{align*}
\mathcal{J}_{a,b}^{\lambda, u_{\varepsilon}}(\zeta)&=\frac{a}{2} \zeta^2+\frac{b}{4}\zeta^{4}-\frac{\zeta^{2^*_s}}{2^*_s} \Vert u_{\varepsilon} \Vert_{2^*_s}^{2^*_s}-\lambda \int_{\Omega}G(x, \zeta u_{\varepsilon})\, dx \\
& \leq \zeta^2 f_{N,s}(\zeta)- \frac{\zeta^{2^*_s}}{2^*_s}O(\varepsilon^{\frac{N-2s}{2}})- \lambda \int_{\Omega}G(x,\zeta u_{\varepsilon})\, dx.
\end{align*}
Choosing as $\zeta=m_{N,s}$ we obtain
\begin{equation} \label{eq25}
\mathcal{J}_{a,b}^{\lambda, u_{\varepsilon}}(m_{N,s})=-\frac{m_{N,s}^{2^*_s}}{2^*_s}O(\varepsilon^{\frac{N-2s}{2}})-\lambda \int_{\Omega}G(x,m_{N,s} u_{\varepsilon}) \, dx.
\end{equation}
\textit{Claim:} There exists a constant $C_2>0$ such that $\int_{\Omega}G(x,m_{N,s} u_{\varepsilon}) \, dx \geq C_2 \varepsilon^{\frac{N}{2}}$ as $\varepsilon \to 0$.

assumptions $H_2$ implies the existence of $\mu>0$ such that
$g(x,t) \geq \chi_I$ where $I$ is an open interval of $(0, \infty)$
and $\chi_I$ is its characteristic function. So, there exists
$\beta>0$ such that
$G(x,t) \geq \tilde{G}(t):= \mu \int_0^t \chi_I(\tau) d\tau \geq
\beta$ for any $t \geq \alpha$ where $\alpha:= \inf I$ is positive. At
this point, we have
\begin{align} \label{eq26}
\int_{\Omega}G(x,m_{N,s}u_{\varepsilon})\, dx & \geq \int_{|x| \leq R} G(x,m_{N,s}u_{\varepsilon})\, dx=\int_{|x| \leq R}G\left(x, \frac{m_{N,s}}{\Vert v_{\varepsilon} \Vert (\varepsilon+|x|^2)^{\frac{N-2s}{2}}} \right)\, dx \notag \\
& \geq \int_{|x| \leq R} \tilde{G} \left(\frac{m_{N,s}}{\Vert v_{\varepsilon}\Vert (\varepsilon + | x|^2)^{\frac{N-2s}{2}}}\, \right) dx \\
& = \int_0^R \tilde{G} \left(\frac{m_{N,s}}{\Vert v_{\varepsilon}\Vert (\varepsilon + w^2)^{\frac{N-2s}{2}}} \right)w^{N-1} \, dw \notag \\
& \geq \int_0^{\sqrt{\varepsilon}R} \tilde{G} \left(\frac{m_{N,s}}{\Vert v_{\varepsilon}\Vert (\varepsilon + w^2)^{\frac{N-2s}{2}}} \right)w^{N-1} \, dw.
\end{align}
We emphasize that if
\[
\frac{m_{N,s}}{\Vert v_{\varepsilon}\Vert (\varepsilon + w^2)^{\frac{N-2s}{2}}} \geq \alpha
\]
then
\[
\int_0^{\sqrt{\varepsilon}R} \tilde{G} \left(\frac{m_{N,s}}{\Vert v_{\varepsilon}\Vert (\varepsilon + w^2)^{\frac{N-2s}{2}}} \right)w^{N-1} \, dw \geq \beta \int_0^{\sqrt{\varepsilon}R} w^{N-1} \, dw= C_2 \varepsilon^{\frac{N}{2}}.
\]
Since $ w \in \left[0, \sqrt{\varepsilon} R\right]$, we have
\begin{align*}
\frac{m_{N,s}}{\Vert v_{\varepsilon}\Vert (\varepsilon+ w^2)^{\frac{N-2s}{2}}} & \geq \frac{m_N,s}{\varepsilon^{-\frac{N-2s}{4}}\left( C_1+O(\varepsilon^{\frac{N-2s}{4}}) \right)\left(\varepsilon+ R^2\varepsilon \right)^{\frac{N-2s}{2}}} \\
&=\frac{m_N,s}{\varepsilon^{\frac{N-2s}{4}}\left( C_1+O(\varepsilon^{\frac{N-2s}{4}}) \right)\left(1+R^2 \right)^{\frac{N-2s}{2}}} \geq \alpha
\end{align*}
as $\varepsilon \to 0$ proving the claim.

As a consequence of the claim and \eqref{eq26} we obtain
\[
\mathcal{J}_{a,b}^{\lambda, u_{\varepsilon}}(m_{N,s})\leq \varepsilon^{\frac{N-2s}{2}} \left( -\frac{1}{2^*_s}O(1)-\lambda C_2 \varepsilon^s \right)<0.
\]
Hence, $\lambda_0^s(u_{\varepsilon})<\lambda$. We can now let $\lambda \to 0$ and we get
$\overline{\lambda}_0^s=0$ as desired.  In order to see the last part,
let $(u_k)_k \subset X_0^s(\Omega) \setminus \lbrace 0 \rbrace $ be a
sequence such that
$\lambda_k:= \lambda_0^s(u_k) \to \overline{\lambda}_0^s=0$. As we did
in part $i)$, we suppose $\Vert u_k \Vert=1$, $u_k \rightharpoonup u$
and that there exists $\zeta_k>0$ such that
\begin{equation} \label{eq27}
\frac{a}{2}+\frac{b}{4}\zeta_k^2+ \frac{\zeta^{2^*_s-2}}{2^*_s} \Vert u_k \Vert_{2^*_s}^{2^*_s}-\lambda_k\int_{\Omega} \frac{G(x,\zeta_k u_k)}{\zeta_k^2} \, dx=0.
\end{equation}
Combining assumptions $H_3$, $H_4$ and \eqref{eq27}, we can deduce that, up to subsequence, $\zeta_k \to \overline{\zeta}$ and $\Vert u_k \Vert_{2^*_s}^{2^*_s}\to \gamma$ as $k \to \infty$. Passing to the limit in \eqref{eq27}, we get
\[
\frac{a}{2}+\frac{b}{4}\overline{\zeta}^2-\frac{\overline{\zeta}^{2^*_s-2}}{2^*_s}\gamma=0.
\]
From $a^{(N-4s)/2s}b=L_{N,s}$ it follows that
$\gamma=S_{N,s}^{-\frac{2^*_s}{2}}$, thus $(u_k)_k$ is a minimizing
sequence for $S_{N,s}$. Now, by contradiction assume $u \neq 0$. We
point out that by the lower semicontinuity of the norm we have
$\Vert u \Vert \leq 1$. Coupling this fact with Remark \ref{rem2}, we
obtain
\begin{multline*}
0 \leq \frac{a}{2}+\frac{b}{4}\overline{\zeta}^2-\frac{S_{N,s}^{-\frac{2^*_s}{2}}}{2^*_s} \overline{\zeta}^{2^*_s-2} \Vert u \Vert ^{2^*_s}  \leq \frac{a}{2}+ \frac{b}{4}\overline{\zeta}^2-\frac{\overline{\zeta}^{2^*_s-2}}{2^*_s} \Vert u \Vert_{2^*_s}^{2^*_s}  \\
 \leq \limsup_{k \to \infty} \left(\frac{a}{2}+\frac{b}{2}\zeta_k^2-\frac{\zeta_k^{2^*_s}}{2^*_s} \Vert u_k \Vert_{2^*_s}^{2^*_s}-\lambda_k \int_{\Omega} \frac{G(x,\zeta_k u_k)}{\zeta_k^2} \, dx \right)
=0,
\end{multline*}
which cannot happen since $\Omega$ is bounded, see~\cite{MR2064421}.
\end{proof}
Next proposition summarize the situation of the infimum depending on
the choice of the parameter $\lambda$ for the functional
$\mathcal{J}_{a,b}^{\lambda,u}(\zeta)$.
\begin{proposition} \label{prop7} If
  $\lambda \leq \overline{\lambda}_0^s $ then
  $ \inf_{\zeta >0} \mathcal{J}_{a,b}^{\lambda,u}(\zeta) =0$ for any
  $u \in X_0^s (\Omega) \setminus \lbrace 0 \rbrace $. On the other
  hand, if $\lambda > \overline{\lambda}_0^s$ there exists
  $ u \in X_0^s (\Omega) \setminus \lbrace 0 \rbrace $ such that
  $ \inf_{\zeta >0} \mathcal{J}_{a,b}^{\lambda,u}(\zeta) <0$.
\end{proposition}
\begin{proof}
  Take $\lambda \leq \overline{\lambda}_0^s $. We have that
  $ \lambda \leq \overline{\lambda}_0^s \leq \lambda_0^s(u)$ for any
  $u \in X_0^s (\Omega) \setminus \lbrace 0 \rbrace $, then the
  conclusion comes from Corollary \ref{cor1}. Instead, let us consider
  $\lambda \in \R^*$ such that $\lambda > \overline{\lambda}_0^s$. By
  the definition of infimum, it is possible to find
  $u \in X_0^s (\Omega) \setminus \lbrace 0 \rbrace $ such that
  $\lambda \geq \lambda_0^s(u) > \overline{\lambda}_0^s$. Again, the
  assertion it is a consequence of Corollary \ref{cor1}.
\end{proof}
After some preliminary results we are ready to study the set of solutions of problem \eqref{eq:Pablambda}. The first step will consists in giving the proof for Theorems \ref{th4} and \ref{th6} providing the existence of global minimizers for $\lambda \geq \overline{\lambda}_0^s$.

\begin{proof}[Proof of Theorem \ref{th4}]
  By the use of assumptions ($H_3$) and ($H_4$) it is easy to verify that
  $\mathcal{I}_{a,b}^{\lambda}$ is coercive. Furthermore, from
  \ref{lemma1} we also have the lower semicontinuity. At this point,
  as a consequence of the well known Weiestrass Theorem, we have that
  the infimum is attained. To conclude, we recall that Proposition
  \ref{prop7} implies the existence of a function in which the
  functional turns out to be negative.
\end{proof}

\begin{proof}[Proof of Theorem \ref{th6}]

  \noindent $(i)$ Let $(\lambda_k)_k \subset \R^+$ a sequence such
  that $\lambda_k \searrow \overline{\lambda}_0^s$. Theorem \ref{th4}
  implies the existence of a sequence
  $(u_k)_k \subset X_0^s(\Omega) \setminus \lbrace 0 \rbrace$ such
  that $\iota_{\lambda_k}^s=\mathcal{I}_{a,b}^{\lambda_k}(u_k)<0$. As
  we did in Proposition \ref{prop6}, after fixing $\varepsilon > 0$ we
  have
\begin{equation} \label{eq28}
|G(x,t)| \leq \frac{\varepsilon}{2}t^2+ \frac{c}{p} |t|^p
\end{equation}
for all $(x,t) \in \Omega \times \R$. Hence
\begin{align} \label{eq29}
&\frac{a}{2} \Vert u_k \Vert^2+\frac{b}{4} \Vert u_k \Vert^4-\frac{1}{2^*_s}\Vert u_k \Vert_{2^*_s}^{2^*_s} < \lambda_k \int_{\Omega} G(x,u_k)\, dx \notag \\
& \leq \lambda_k \left( \frac{\varepsilon}{2} \Vert u_k \Vert_2^2+\frac{c}{p} \Vert u_k \Vert_p^p \right) \leq \tilde{C} \left( \Vert u_k \Vert^2 + \Vert u_k \Vert^p \right)
\end{align}
for some $\tilde{C}>0$ since
$X_0^s(\Omega) \hookrightarrow L^q(\Omega)$ continuously for any
$q \in \left[ 2,2^*_s \right]$. From $4>2^*_s$ it follows that
$\left( \Vert u_k \Vert \right)_k $ must be bounded and it is not
restrictive to assume $u_k \rightharpoonup u $ in
$X_0^s(\Omega)$. Applying Lemma \ref{lemma1}[$(1)$] we obtain
\[
\mathcal{I}_{a,b}^{\overline{\lambda}_0^s}(u) \leq \liminf_{k \to \infty} \mathcal{I}_{a,b}^{\overline{\lambda}_0^s}(u_k) \leq 0.
\]
On the other hand , Proposition \ref{prop7} states that
$\mathcal{I}_{a,b}^{\overline{\lambda}_0^s}(v) \geq 0$ for any
$v \in X_0^s(\Omega)$, and so
\begin{equation} \label{eq30}
\iota_{\overline{\lambda}_0^s}^s= \mathcal{I}_{a,b}^{\overline{\lambda}_0^s}(u)=0.
\end{equation}
It remains only to prove that $u$ is a non trivial minimizer. To see
that, observe that
\[
\frac{a}{2}\Vert u_k \Vert^2+\frac{b}{4} \Vert u_k \Vert^4-\frac{S_{N,s}^{-\frac{2^*_s}{2}}}{2^*_s} \Vert u_k \Vert^{2^*_s} \leq \frac{a}{2}\Vert u_k \Vert^2+\frac{b}{4} \Vert u_k \Vert^4-\frac{1}{2^*_s} \Vert u_k \Vert^{2^*_s}_{2^*_s}< \lambda_k < \lambda_k \int_{\Omega}G(x,u_k) \, dx
\]
where we used the fractional Sobolev inequality. Dividing by $\Vert u_k \Vert^2$ and exploiting \eqref{eq28}, we get
\[
f_{N,s}(\Vert u_k \Vert) \leq \lambda_k \left( \frac{\varepsilon}{2} + \frac{c}{p} \Vert u_k \Vert_p^p \right).
\]
Were $u=0$, recalling that $X_0^s(\Omega)\hookrightarrow L^q(\Omega) $
for any $q \in \left[ 2,2^*_s \right)$, we would have
\begin{equation*}
f_{N,s}\left(\Vert u_k \Vert\right) \to 0
\end{equation*}
as $k \to \infty$ since $\varepsilon>0$ is arbitrary. This fact is in
contradiction with
\[
f_{N,s}\left( \Vert u_k \Vert \right) \geq f_{N,s}(m_{N,s})>0
\]
since $a^{(N-2s)/2s}b>L_{N,s}$. So $u$ must be different from zero.

\noindent $(ii)$ From Proposition \ref{prop6}[$(ii)$] we have
$\overline{\lambda}_0^s$, and so
\[
\mathcal{I}_{a,b}^{\overline{\lambda}_0^s}(u)= \frac{a}{2}\Vert u \Vert^2+\frac{b}{4} \Vert u \Vert^4-\frac{1}{2^*_s} \Vert u \Vert_{2^*_s}^{2^*_s}.
\]
In virtue of Remark \ref{rem2}, we have
\[
\mathcal{I}_{a,b}^{\overline{\lambda}_0^s}(u)= \Vert u \Vert^2 f_{N,s}(\Vert u \Vert)>0
\]
for any $u \in X_0^s \setminus \lbrace 0\rbrace$. Since \eqref{eq30}
is still valid, we have that the infimum can be attained only in the
case in which $u=0$.
\end{proof}
\begin{corollary} \label{cor2} If $a^{(N-4s)/2s}b>L_{N,s}$ and
  $u \in X_0^s (\Omega) \setminus \lbrace 0 \rbrace$ is such that
  $\iota_{\overline{\lambda}^s}=
  \mathcal{I}_{a,b}^{\overline{\lambda}_0^s}(u)$ then
  $\overline{\lambda}_0^s=\lambda_0^s(u)$.
\end{corollary}
\begin{proof}
  The pair $(\overline{\lambda}_0^s, u)$ solve the system
  \eqref{eq23}. The conclusion follows by uniqueness.
\end{proof}

\begin{proof}[Proof of Theorem \ref{th5}]
  Fix $\varepsilon>0$ and recall the Aubin-Talenti functions
  $u_{\varepsilon}$ defined in Proposition \ref{prop6}. Choose
  $\zeta >0$ and keep $\lambda >0$ free. We have
\begin{align*}
\mathcal{J}_{a_k,b_k}^{\lambda, u_{\varepsilon}}(\zeta)&= \frac{a_k}{2}\zeta^2+\frac{b_k}{4}\zeta^4-\frac{\zeta^{2^*_s}}{2^*_s} \Vert u_k \Vert_{2^*_s}^{2^*_s}-\lambda \int_{\Omega}G(x,\zeta u_{\varepsilon})\, dx \\
&=\zeta^2 f_{N,s}^k (\zeta ) - \frac{\zeta^{2^*_s}}{2^*_s}O\left(\varepsilon^{\frac{N-2s}{2}} \right)-\lambda \int_{\Omega}G(x,\zeta u_{\varepsilon})\, dx
\end{align*}
where we defined with $f_{N,s}^k$ the map $f_{N,s}$ depending on the
parameters $a_k$, $b_k$. We select $\zeta=m_{N,s}^k$ (here $m_{N,s}^k$
is the point point in which $f_{N,s}^k$ attains its minimum), and
since $m_{N,s}^k \to m_{N,s}$ as $k \to \infty$, we get
\begin{equation} \label{eq31}
\lim_{k \to \infty} \mathcal{J}_{a_k,b_k}^{\lambda, u_{\varepsilon}}(m_{N,s}^k)= -\frac{m_{N,s}^{2^*_s}}{2^*_s}O\left( \varepsilon^{\frac{N-2s}{2}} \right)-\lambda \int_{\Omega} G(x,m_{N,s}u_{\varepsilon})\, dx.
\end{equation}
Recalling that in Proposition \ref{prop6} we obtained the estimate
\[
\int_{\Omega}G(x,m_{N,s}u_{\varepsilon})\, dx \geq C_2 \varepsilon^{\frac{N}{2}},
\]
from \eqref{eq31} we deduce
\[
\lim_{k \to \infty} \mathcal{J}_{a_k,b_k}^{\lambda,u_{\varepsilon}}(m_{N,s}^k) \leq \varepsilon^{\frac{N-2s}{2}} \left( - \frac{1}{2^*_s}O(1) - \lambda C_2 \varepsilon^s \right).
\]
Hence for $k$ sufficiently large and small $\varepsilon$
\[
\mathcal{J}_{a_k,b_k}^{\lambda,u_{\varepsilon}}(m_{N,s}^k) <0.
\]
As a consequence of that, we have that
$\lambda_k \leq \lambda_0^s(u_{\varepsilon}) \leq \lambda$. Letting
$\lambda \to 0$ we obtain that $\lambda_k \to 0$ as $k \to
\infty$. Now, exploiting the homogeneity of degree zero of the
function $\lambda_0^s(\cdot)$ proved in Proposition \ref{prop6}, we
suppose $\Vert u_k \Vert=1$ and $u_k \rightharpoonup u$. Arguing
similarly as we did to get \eqref{eq27} we are able to deduce the
existence of $\zeta_k>0$ such that
\begin{equation} \label{eq32}
\frac{a_k}{2}+\frac{b_k}{4} \zeta_k^2-\frac{\zeta^{2^*_s-2}}{2^*_s} \Vert u_k \Vert^{2^*_s}_{2^*_s}-\lambda_k \int_{\Omega} \frac{G(x,\zeta_k u_k)}{\zeta_k^2} \, dx=0
\end{equation}
Moreover, from $H_3$, $H_4$ and \eqref{eq32} it follows that $\zeta_k \to \overline{\zeta}>0$ and that $\Vert u_k \Vert_{2^*_s}^{2^*s} \\to \gamma $ up to a subsequence as $k \to \infty$. Thus, passing to the limit in \eqref{eq32} we get
\[
\frac{a}{2}+\frac{b}{4} \overline{\zeta}^2-\frac{1}{2^*_s}\gamma \overline{\zeta}^{2^*_s-2}=0.
\]
Since $a^{(N-4s)/2s}b =L_{N,s}$ it must be $\gamma = S_{N,s}$ and that
means $(u_k)_k$ is a minimizing sequence for the optimal Sobolev
constant. We also have that $u=0$. Indeed, if $ u \neq 0$, combining
Remark \ref{rem2}, the fact that by the sequentially lower
semicontinuity of the norm $\Vert u \Vert \leq 1$ and Lemma
\ref{lemma1}[$(1)$], we obtain
\begin{align*}
0& \leq \frac{a}{2}+\frac{b}{4}\overline{\zeta}^2-\frac{S_{N,s}^{-\frac{2^*_s}{2}}}{2^*_s} \overline{\zeta}^{2^*_s-2} \Vert u \Vert^{2^*_s} \leq \frac{a}{2}+\frac{b}{4}\overline{\zeta}^2- \frac{\overline{\zeta}^{2^*_s-2}}{2^*_s} \Vert u \Vert^{2^*_s}_{2^*_s} \\
& \leq \liminf_{k \to \infty} \left( \frac{a_k}{2}+\frac{b_k}{4}\zeta_k^2- \frac{\zeta_k^{2^*_s-2}}{2^*_s} \Vert u_k \Vert^{2^*_s}_{2^*_s} -\lambda_k \int_{\Omega} \frac{ G(x,\zeta_k u_k)}{\zeta_k^2} \, dx \right)=0
\end{align*}
The conclusion comes from the nonexistence of minimizers for $S_{N,s}$
in bounded sets as shown in \cite{MR2064421}.
\end{proof}
Now, we begin to investigate solutions of mountain pass type. As we
will see, the situation changes if
$\lambda \geq \overline{\lambda}_0^s$ or
$\lambda < \overline{\lambda}_0^s$. The reader should keep in mind
that from now to the end of the section we will consider positive
parameters $a, b \in \R$ such that $a^{(N-4s)/2s}b > L_{N,s}$.

\begin{proof} [Proof of theorem \ref{th7}]
  Take $\varepsilon >0$. Recalling \eqref{eq28} and that
  $X_0^s(\Omega) \hookrightarrow L^q(\Omega)$ continuously for any
  $q \in \left[ 2, 2^*_s \right]$ we obtain
\begin{equation} \label{eq33}
\mathcal{I}_{a,b}^{\lambda} \geq \left( \frac{a}{2}-\lambda C \varepsilon \right) \Vert u \Vert^2+\frac{b}{4} \Vert u \Vert^4-C \Vert u \Vert^{2^*_s}-\lambda C \Vert u \Vert^p
\end{equation}
where $C>0$ is a constant chosen adequately. By selecting $ \varepsilon < a/(2 \lambda C)$ there exists $R_{\lambda}^s$ such that
\[
\inf_{\Vert u \Vert = R^s_{\lambda}} \mathcal{I}_{a,b}^{\lambda} >0.
\]
Now, observe that $\mathcal{I}_{a,b}^{\lambda}(0)=0$ and
$\mathcal{I}_{a,b}^{\lambda}(u_{\overline{\lambda}_0^s}^s) \leq
0$. Indeed,
$\mathcal{I}_{a,b}^{\lambda}(u_{\overline{\lambda}_0^s}^s) = 0$ if
$\lambda = \overline{\lambda}_0^s$ while
$\mathcal{I}_{a,b}^{\lambda}(u_{\overline{\lambda}_0^s}^s) < 0$ for
$\lambda > \overline{\lambda}_0^s$ by Proposition \ref{prop7}. As a
consequence of that, the functional possesses a mountain pass
geometry. Furthermore, recalling Lemma~\ref{lemma1}[$(2)$] we have that
$\mathcal{I}_{a,b}^{\lambda}$ satisfies the Palais-Smale condition. At
this point the conclusion is obtained by applying the classic mountain
pass theorem.
\end{proof}
Having analysed the situation for
$\lambda \geq \overline{\lambda}_0^s$, now we draw our attention to
the case $\lambda < \overline{\lambda}_0^s$. Namely, we will show the
existence of non trivial solutions that are local minimizer or of
mountain pass type.
\begin{proposition} \label{prop8} If
  $\lambda \leq \overline{\lambda}_0^s$ then it is possible to find
  $r=r(s), M=M(s)>0$ such that
 \begin{equation} \label{eq34}
 \inf \left\lbrace \mathcal{I}_{a,b}^{\lambda}(u): u \in X_0^s(\Omega), \ \Vert u \Vert=r \right\rbrace \geq M.
\end{equation}
\end{proposition}
\begin{proof}
Given $\varepsilon >0$, because $\lambda \leq \overline{\lambda}_0^s$ and recalling \eqref{eq33}, we get
\[
\mathcal{I}_{a,b}^{\lambda} \geq \left( \frac{a}{2}-\overline{\lambda}_0^s C \varepsilon \right) \Vert u \Vert^2+\frac{b}{4} \Vert u \Vert^4-C \Vert u \Vert^{2^*_s}-\overline{\lambda}_0^s C \Vert u \Vert^p
\]
for any $u \in X_0^s(\Omega)$. The statement follows by taking
$\varepsilon$ so that $a/2-\overline{\lambda}_0^sC\varepsilon>0$.
\end{proof}
Now,  we are able to characterize the infimum in Theorem \ref{th8}. More precisely, considering the $r>0$ given by the previous Proposition, we can set
\[
\hat{\iota}_{\lambda}^s:= \inf \lbrace \mathcal{I}_{a,b}^{\lambda}(u):u \in X_0^s(\Omega), \ \Vert u \Vert \geq r\rbrace.
\]
\begin{remark} \label{rem5} It is straightforward to see that
  $\hat{\iota}_0^s \to 0$ as $\lambda \to \overline{\lambda}_0^s$. Indeed, it suffices to consider a function $u \in X_0^s(\Omega)$ such
  that $\overline{\lambda}_0^s=\lambda_0^s(u)$ (the existence of a
  such $u$ is guaranteed by Theorem \ref{th6}) and to observe that
\[
  0 \leq \hat{\iota}_{\lambda}^s\leq \mathcal{I}_{a,b}^{\lambda}(u)
  \to 0 \quad \mbox{as} \quad \lambda \to \overline{\lambda}_0^s.
\]
\end{remark}

\begin{remark} \label{rem6}
The function $w_{\lambda}^s$ obtained in the previous Theorem represents a critical point for the functional $\mathcal{I}_{a,b}^{\lambda}$ and it is a local minimizer.
\end{remark}
\begin{proof}[Proof of Theorem \ref{th8}]
Consider the $r, M>0$ given by Proposition \ref{prop8} and notice that if $ \lambda \in (\overline{\lambda}_0^s-\delta, \overline{\lambda}_0^s )$ we have that $\hat{\iota}_{\lambda}^s<M$ for an appropriate $\delta>0$. As a consequence of that, if $(u_k)_k$ is a minimizing sequence there must be a $\upsilon>0$ such that $\Vert u_k \Vert \geq M + \upsilon$ for $k$ sufficiently large. At this point, invoking the Ekeland's variational principle (see \cite{zbMATH03449362}) we have the existence of a minimizing sequence and the convergence to a local minimizer $w_\lambda^s \in X_0^s(\Omega)$ such that $ \Vert w_{\lambda}^s \Vert >M$ and $\hat{\iota}_{\lambda}^s=\mathcal{I}_{a,b}^{\lambda}(w_{\lambda}^s)$ is established remembering the validity of the Palais-Smale condition as showed in Lemma \ref{lemma1}[$(2)$].
\end{proof}
Finally, we prove Theorem \ref{th9} that ensure the existence of mountain pass solutions for $\lambda < \overline{\lambda}_0^s$ close enough to $\overline{\lambda}_0^s$. In the following we will denote with $\delta>0$ the number obtained in Theorem \ref{th8}.

\begin{proof}[Proof of Theorem \ref{th9}]
Notice $\min \lbrace \mathcal{I}_{a,b}^{\lambda}(0), \mathcal{I}_{a,b}^{\lambda}(w_{\lambda}^s)\rbrace<M$, recall $\Vert w_{\lambda}^s \Vert >M$ and \eqref{eq34}. So, we have a mountain pass geometry. Since the Palais-Smale condition is satisfied, we exploit the the Mountain Pass Theorem (see \cite{zbMATH03429675}) to get the conclusion.
\end{proof}
\begin{remark}
If in addition to assumptions $(H_1)-(H_4)$ we require
\begin{itemize}
\item[($H_5$)] for any $u \in X_0^s(\Omega)$, the function $\zeta \mapsto \displaystyle\int_{\Omega} g(x,\zeta u(x)) \, dx $ is $C^1$ on $(0,\infty )$
\end{itemize}
we are able to state a non-existence result for problem \eqref{eq:Pablambda}. Namely, we claim there is $\overline{\lambda}^s:= \overline{\lambda}^s(a,b) \in (0,\overline{\lambda}_0^s)$ such that if $\lambda \in (0, \overline{\lambda}^s)$ then \eqref{eq:Pablambda} does not admit non trivial solutions. Consider the system
\begin{equation} \label{eq36}
\begin{cases}
\left(\mathcal{J}_{a,b}^{\lambda,u}\right)'(\zeta)=0 \\
\left(\mathcal{J}_{a,b}^{\lambda,u}\right)''(\zeta)=0 \\
\left(\mathcal{J}_{a,b}^{\lambda,u}\right)'(\zeta)=\inf_{\varrho>0} \left(\mathcal{J}_{a,b}^{\lambda,u}\right)'(\varrho).
\end{cases}
\end{equation}
After fixing $u \in X_0^s(\Omega)$, similarly to proposition \ref{prop3} it is possible to find a unique $\lambda^s(u) >0$ that solves \eqref{eq36}. We point out that the parameter $\lambda^s(u)$ is the unique $\lambda >0$ for which the fiber map $\mathcal{J}_{a,b}^{\lambda,u}(\zeta)$ has a critical point with null second derivative. Furthermore, observe that if $0 < \lambda < \lambda^s(u)$ then
\[
\mathcal{J}_{a,b}^{\lambda,u}(\zeta) > \mathcal{J}_{a,b}^{\lambda^s(u),u}(\zeta)>0.
\]
So, $\mathcal{J}_{a,b}^{\lambda,u}(\zeta)$ has no critical points. As a consequence of that, it is immediate to prove that
\begin{equation}\label{eq37}
\lambda^s(u) < \lambda^s_0(u).
\end{equation}
If not, we would have that $\mathcal{J}_{a,b}^{\lambda_0^s(u),u}(\zeta)$ is increasing contradicting the existence of solutions for system \eqref{eq23}. At this point we set
\[
\overline{\lambda}^s:= \inf_{u \in X_0^s(\Omega) \setminus \lbrace 0 \rbrace} \lambda^s(u).
\]
Now, we observe that if $a^{(N-4s)/2s}b > \mathsf{PS}_{N,s}$ then $0<\overline{\lambda}^s<\overline{\lambda}_0^s$. In fact, we know from Corollary \ref{cor2} that there is $u \in  X_0^s(\Omega) \setminus \lbrace 0 \rbrace$ such that $\overline{\lambda}_0^s=\lambda_0^s(u)$. From \eqref{eq37} it follows that
\[
\overline{\lambda}^s \leq \lambda^s(u) < \lambda_0^s(u)=\overline{\lambda}_0^s.
\]
To conclude, we observe that for any $\lambda \in (0,\overline{\lambda}^s )$ the map $\mathcal{J}_{a,b}^{\lambda,u}(\zeta)$ is increasing and $\left( \mathcal{J}_{a,b}^{\lambda,u} \right)'(\zeta) >0$ for all $\zeta >0$. Hence $u=0$ is the only admissible critical point.
\end{remark}

\bibliographystyle{amsplain}
\bibliography{FFK}

\providecommand{\bysame}{\leavevmode\hbox to3em{\hrulefill}\thinspace}
\providecommand{\MR}{\relax\ifhmode\unskip\space\fi MR }
\providecommand{\MRhref}[2]{%
  \href{http://www.ams.org/mathscinet-getitem?mr=#1}{#2}
}
\providecommand{\href}[2]{#2}
\begin{thebibliography}{10}

\bibitem{MR3967804}
Nicola Abatangelo and Enrico Valdinoci, \emph{Getting acquainted with the
  fractional {L}aplacian}, Contemporary research in elliptic {PDE}s and related
  topics, Springer INdAM Ser., vol.~33, Springer, Cham, 2019, pp.~1--105.
  \MR{3967804}

\bibitem{zbMATH03429675}
Antonio {Ambrosetti} and Paul~H. {Rabinowitz}, \emph{{Dual variational methods
  in critical point theory and applications}}, {J. Funct. Anal.} \textbf{14}
  (1973), 349--381 (English).

\bibitem{MR4018100}
Vincenzo Ambrosio, \emph{Concentrating solutions for a class of nonlinear
  fractional {S}chr\"{o}dinger equations in {$\Bbb R^N$}}, Rev. Mat. Iberoam.
  \textbf{35} (2019), no.~5, 1367--1414. \MR{4018100}

\bibitem{MR3373607}
Giuseppina Autuori, Alessio Fiscella, and Patrizia Pucci, \emph{Stationary
  {K}irchhoff problems involving a fractional elliptic operator and a critical
  nonlinearity}, Nonlinear Anal. \textbf{125} (2015), 699--714. \MR{3373607}

\bibitem{MR699419}
Ha\"{\i}m Br\'{e}zis and Elliott Lieb, \emph{A relation between pointwise
  convergence of functions and convergence of functionals}, Proc. Amer. Math.
  Soc. \textbf{88} (1983), no.~3, 486--490. \MR{699419}

\bibitem{MR2675483}
L.~Caffarelli, J.-M. Roquejoffre, and O.~Savin, \emph{Nonlocal minimal
  surfaces}, Comm. Pure Appl. Math. \textbf{63} (2010), no.~9, 1111--1144.
  \MR{2675483}

\bibitem{MR2540182}
Luis Caffarelli, \emph{Surfaces minimizing nonlocal energies}, Atti Accad. Naz.
  Lincei Rend. Lincei Mat. Appl. \textbf{20} (2009), no.~3, 281--299.
  \MR{2540182}

\bibitem{MR3289358}
\bysame, \emph{Non-local diffusions, drifts and games}, Nonlinear partial
  differential equations, Abel Symp., vol.~7, Springer, Heidelberg, 2012,
  pp.~37--52. \MR{3289358}

\bibitem{MR2588587}
Luis~A. Caffarelli, \emph{Some nonlinear problems involving non-local
  diffusions}, I{CIAM} 07---6th {I}nternational {C}ongress on {I}ndustrial and
  {A}pplied {M}athematics, Eur. Math. Soc., Z\"{u}rich, 2009, pp.~43--56.
  \MR{2588587}

\bibitem{MR2064421}
Athanase Cotsiolis and Nikolaos~K. Tavoularis, \emph{Best constants for
  {S}obolev inequalities for higher order fractional derivatives}, J. Math.
  Anal. Appl. \textbf{295} (2004), no.~1, 225--236. \MR{2064421}

\bibitem{MR2944369}
Eleonora Di~Nezza, Giampiero Palatucci, and Enrico Valdinoci,
  \emph{Hitchhiker's guide to the fractional {S}obolev spaces}, Bull. Sci.
  Math. \textbf{136} (2012), no.~5, 521--573. \MR{2944369}

\bibitem{zbMATH03449362}
I.~{Ekeland}, \emph{{On the variational principle}}, {J. Math. Anal. Appl.}
  \textbf{47} (1974), 324--353.

\bibitem{MR4201645}
F.~Faraci and K.~Silva, \emph{On the {B}rezis-{N}irenberg problem for a
  {K}irchhoff type equation in high dimension}, Calc. Var. Partial Differential
  Equations \textbf{60} (2021), no.~1, 22. \MR{4201645}

\bibitem{faraci2018energy}
Francesca Faraci, Csaba Farkas, and Alexandru Kristály, \emph{Energy
  properties of critical {K}irchhoff problems with applications}, 2018.

\bibitem{MR3120682}
Alessio Fiscella and Enrico Valdinoci, \emph{A critical {K}irchhoff type
  problem involving a nonlocal operator}, Nonlinear Anal. \textbf{94} (2014),
  156--170. \MR{3120682}

\bibitem{zbMATH02674345}
G.~{Kirchhoff}, \emph{{Vorlesungen \"uber mathematische Physik. Erster Band:
  Mechanik. Vierte Auflage. Herausgegeben von W. Wien. Mit 18 Figuren im
  Text}}, {Leipzig: B. G. Teubner. X und 464 S. gr. \(8^\circ\)}, 1897.

\bibitem{MR834360}
P.-L. Lions, \emph{The concentration-compactness principle in the calculus of
  variations. {T}he limit case. {I}}, Rev. Mat. Iberoamericana \textbf{1}
  (1985), no.~1, 145--201. \MR{834360}

\bibitem{MR850686}
\bysame, \emph{The concentration-compactness principle in the calculus of
  variations. {T}he limit case. {II}}, Rev. Mat. Iberoamericana \textbf{1}
  (1985), no.~2, 45--121. \MR{850686}

\bibitem{MR3679329}
Zhisu Liu, Marco Squassina, and Jianjun Zhang, \emph{Ground states for
  fractional {K}irchhoff equations with critical nonlinearity in low
  dimension}, NoDEA Nonlinear Differential Equations Appl. \textbf{24} (2017),
  no.~4, Paper No. 50, 32. \MR{3679329}

\bibitem{MR3961733}
Xiang Mingqi, Vicen\c{t}iu~D. R\u{a}dulescu, and Binlin Zhang, \emph{A critical
  fractional {C}hoquard-{K}irchhoff problem with magnetic field}, Commun.
  Contemp. Math. \textbf{21} (2019), no.~4, 1850004, 36. \MR{3961733}

\bibitem{zbMATH06533015}
Giovanni {Molica Bisci}, Vicentiu~D. {Radulescu}, and Raffaella {Servadei},
  \emph{{Variational methods for nonlocal fractional problems}}, vol. 162,
  Cambridge: Cambridge University Press, 2016 (English).

\bibitem{MR3575909}
Giovanni Molica~Bisci and Luca Vilasi, \emph{On a fractional degenerate
  {K}irchhoff-type problem}, Commun. Contemp. Math. \textbf{19} (2017), no.~1,
  1550088, 23. \MR{3575909}

\bibitem{MR3216834}
Giampiero Palatucci and Adriano Pisante, \emph{Improved {S}obolev embeddings,
  profile decomposition, and concentration-compactness for fractional {S}obolev
  spaces}, Calc. Var. Partial Differential Equations \textbf{50} (2014),
  no.~3-4, 799--829. \MR{3216834}

\bibitem{zbMATH07081235}
Patrizia {Pucci} and Vicen\c{t}iu~D. {R\u{a}dulescu}, \emph{{Progress in
  nonlinear Kirchhoff problems}}, {Nonlinear Anal., Theory Methods Appl., Ser.
  A, Theory Methods} \textbf{186} (2019), 1--5 (English).

\bibitem{MR2879266}
Raffaella Servadei and Enrico Valdinoci, \emph{Mountain pass solutions for
  non-local elliptic operators}, J. Math. Anal. Appl. \textbf{389} (2012),
  no.~2, 887--898. \MR{2879266}

\bibitem{MR3271254}
\bysame, \emph{The {B}rezis-{N}irenberg result for the fractional {L}aplacian},
  Trans. Amer. Math. Soc. \textbf{367} (2015), no.~1, 67--102. \MR{3271254}

\bibitem{MR1400007}
Michel Willem, \emph{Minimax theorems}, Progress in Nonlinear Differential
  Equations and their Applications, vol.~24, Birkh\"{a}user Boston, Inc.,
  Boston, MA, 1996. \MR{1400007}

\end{thebibliography}
\end{document}